\documentclass[a4paper]{amsart}
\usepackage[english]{babel}

\usepackage{latexsym}
\usepackage{amssymb}
\usepackage{amsmath}
\usepackage{amsfonts}
\usepackage{amsthm}
\usepackage{xspace}
\usepackage{mathtools}

\usepackage{aliascnt} 
\usepackage{fixltx2e}
\usepackage[usenames, dvipsnames]{xcolor}

\usepackage{eucal}
\usepackage{mathrsfs}   

\usepackage[colorlinks=true,linkcolor=blue,citecolor=blue,urlcolor=blue,citebordercolor={0 0 1},urlbordercolor={0 0 1},linkbordercolor={0 0 1}]{hyperref}
\usepackage[nameinlink]{cleveref}
\crefname{enumi}{part}{parts}

\usepackage{enumitem}
\setlist[enumerate]{font=\normalfont, label=\upshape{(\arabic*)}}
\let\itemref\ref

\makeatletter
\newcommand{\myitem}[1][]{\item[(#1)]\refstepcounter{enumi}\def\@currentlabel{\upshape{(#1)}}}
\makeatother

\usepackage[all,matrix,arrow,cmtip,2cell]{xy} 
\SelectTips{cm}{}
\newdir{(}{{}*!/-5pt/@^{(}}
\newdir{(x}{{}*!/-5pt/@_{(}}
\newdir{+}{{}*!/-9pt/{}}
\newdir{>+}{@{>}*!/-9pt/{}}
\entrymodifiers={+!!<0pt,\fontdimen22\textfont2>}
\UseTwocells

\theoremstyle{plain}
\newtheorem{theorem}{Theorem}[section]

\newtheorem{lemma}[theorem]{Lemma}
\newtheorem{proposition}[theorem]{Proposition}

\theoremstyle{definition}
\newtheorem{definition}[theorem]{Definition}
\newtheorem{example}[theorem]{Example}

\theoremstyle{remark}
\newtheorem{remark}[theorem]{Remark}

\newcommand{\QQ}{\mathbb{Q}}
\newcommand{\ZZ}{\mathbb{Z}}

\newcommand{\VV}{\mathbb{V}} 
\newcommand{\LL}{\mathbb{L}} 
\newcommand{\inj}{\hookrightarrow}
\newcommand{\surj}{\twoheadrightarrow}
\newcommand{\sI}{\mathcal{I}} 
\newcommand{\sJ}{\mathcal{J}} 
\newcommand{\sE}{\mathcal{E}} 
\newcommand{\sN}{\mathcal{N}} 
\newcommand{\sO}{\mathcal{O}} 
\newcommand{\sP}{\mathcal{P}} 
\newcommand{\stG}{\mathscr{G}} 
\newcommand{\stX}{\mathscr{X}} 
\newcommand{\stY}{\mathscr{Y}} 
\newcommand{\stZ}{\mathscr{Z}} 
\newcommand{\stU}{\mathscr{U}} 
\newcommand{\stV}{\mathscr{V}} 
\newcommand{\stW}{\mathscr{W}} 
\newcommand{\GL}{\mathrm{GL}} 
\DeclareMathOperator{\Spec}{Spec}
\DeclareMathOperator{\Isom}{Isom}
\DeclareMathOperator{\Sym}{Sym} 
\DeclareMathOperator{\Gr}{Gr} 

\newcommand{\SEC}{\mathrm{SEC}}
\newcommand{\ET}{\mathrm{\acute{E}T}}

\newcommand{\cl}{\mathrm{cl}} 
\newcommand{\heart}{\heartsuit} 
\DeclareMathOperator{\Map}{Map} 

\DeclareMathOperator{\coker}{coker}
\DeclareMathOperator{\Hom}{Hom}
\DeclareMathOperator{\Ext}{Ext}
\let\im\relax
\DeclareMathOperator{\im}{im} 

\newcommand{\Sch}{\mathsf{Sch}}
\newcommand{\Set}{\mathsf{Set}}
\newcommand{\Coh}{\mathsf{Coh}}
\newcommand{\QCoh}{\mathsf{QCoh}}
\newcommand{\op}{\mathrm{op}}
\newcommand{\DQCOH}{\mathrm{D}_{\QCoh}}

\newcommand{\spref}[1]{\href{http://stacks.math.columbia.edu/tag/#1}{#1}}
\newcommand{\spcite}[1]{\cite[\spref{#1}]{stacks-project}}

\newcommand{\loccit}{\emph{loc.\ cit.}}

\newcommand{\cI}{\mathcal{I}}
\newcommand{\cE}{\mathcal{E}}
\newcommand{\cN}{\mathcal{N}}
\newcommand\oh{\mathcal{O}}
\newcommand\co{\colon}
\newcommand{\iso}{\stackrel{\sim}{\to}}
\newcommand{\tensor}{\otimes}
\newcommand{\Coll}{\mathscr{C}}

\newcommand{\AR}[1]{\mathrm{(AR)_{#1}}} 

\begin{document}

\title[Artin algebraization for pairs]{Artin algebraization for pairs with applications to the local structure of stacks and Ferrand pushouts}
\author[J. Alper]{Jarod Alper}
\address{Department of Mathematics\\University of Washington\\Box 354350\\Seattle, WA 98195-4350\\USA}
\email{jarod@uw.edu}
\author[D. Halpern-Leistner]{Daniel Halpern-Leistner}
\address{Malott Hall, Mathematics Dept\\Cornell University\\Ithaca, NY 14853\\USA}
\email{daniel.hl@cornell.edu}
\author[J. Hall]{Jack Hall}
\address{School of Mathematics \& Statistics\\The University of Melbourne\\Parkville,
  VIC, 3010\\Australia}
\email{jack.hall@unimelb.edu.au}
\author[D. Rydh]{David Rydh}
\address{KTH Royal Institute of Technology\\Department of
  Mathematics\\SE-100 44 Stockholm\\Sweden}
\email{dary@math.kth.se}
\thanks{The first author was partially supported by NSF grant
  DMS-1801976 and DMS-2100088. The second author was partially supported by NSF grants
  DMS-1945478 (CAREER), DMS-1601967, a Simons Collaboration Grant and
  the Alfred P.~Sloan Foundation. The third author was partially
  supported by the Australian Research Council DP210103397. The fourth
  author was supported by the Swedish Research Council 2015-05554 and
  the G\"oran Gustafsson Foundation for Research in Natural Sciences
  and Medicine.}
\date{May 17, 2022}
\subjclass[2020]{14B12, 14D23, 13B12}
\keywords{Artin algebraization, local structure of stacks, Milnor squares,
Ferrand pushouts, derived stacks, henselization}

\begin{abstract}
We give a variant of Artin algebraization along closed subschemes and
closed substacks. Our main application is the existence of \'etale, smooth, or syntomic neighborhoods
of closed subschemes and closed substacks. In particular, we prove
local structure theorems for stacks and their derived counterparts and the existence
of henselizations along linearly fundamental closed substacks. These results establish the existence of Ferrand pushouts, which answers positively a question of Temkin--Tyomkin.
\end{abstract}

\maketitle



\begin{section}{Introduction}
The main technical result of this paper is a generalization of Artin's
algebraization theorem~\cite[Thm.~1.6]{artin_alg_formal_moduli_I}: from
algebraizations of complete local rings to algebraizations
of rings complete along an ideal. It is proven using Artin approximation
over henselian pairs following the approach of \cite{conrad-deJong} and
\cite[App.~A]{alper-hall-rydh_luna-stacks}.

\begin{theorem}[Artin algebraization for pairs] \label{T:artin-algebraization-intro}
Let $S$ be an excellent affine scheme and let $\stX$ be a category fibered in
groupoids, locally of finite presentation over $S$. Let $Z$ be an affine
scheme over $S$, complete along a closed subscheme $Z_0$. Assume that $Z_0\to S$ is of
finite type. Let $\eta\colon Z\to \stX$ be a morphism, formally versal at
$Z_0$. Then there exist
\begin{enumerate}
\item an affine scheme $W$ of finite type over $S$,
\item a closed subscheme $W_0\inj W$,
\item a morphism $\xi\colon W\to \stX$ over $S$ and 
\item a morphism $\varphi\colon (Z,Z_0)\to (W,W_0)$ over $S$
\end{enumerate}
such that the induced morphism $\widehat{\varphi}\colon Z\to \widehat{W}$
is an isomorphism and the isomorphism $\varphi_n\colon Z_n\to W_n$ on
infinitesimal
neighborhoods is compatible with $\eta$ and $\xi$ for every~$n$.
\end{theorem}

We prove a more general version when $Z$ is a stack in \Cref{T:Artin-algebraization}. This 
generalizes~\cite[App.~A]{alper-hall-rydh_luna-stacks} and is used to establish a local structure theorem for stacks (\Cref{T:local-structure:excellent}). We will return to this shortly.

\subsection*{Application: \'Etale neighborhoods of affine subschemes}
As an application of \Cref{T:artin-algebraization-intro} we have the existence of affine \'etale neighborhoods.

\begin{theorem}[Affine \'etale neighborhoods]\label{T:affine-etale-nbhds}
Let $\stX$ be a quasi-separated algebraic stack with affine stabilizers,
and consider a diagram
\[\xymatrix{
  W_0 \ar@{^(-->}[r] \ar[d]^{f_0}  & W \ar@{-->}[d]^f \\
  \stX_0 \ar@{^(->}[r]  & \stX,   
}\]
where $\stX_0\inj \stX$ is a closed immersion and $f_0\colon W_0\to \stX_0$ is an \'etale (resp.\ smooth) morphism with $W_0$ affine. 
Then there exist an affine scheme $W$ and an \'etale
(resp.\ smooth) morphism $f\colon W\to \stX$ such that $f|_{\stX_0}=f_0$. 
\end{theorem}

If $\stX$ is an affine scheme, then \Cref{T:affine-etale-nbhds} is
\spcite{04D1} ($f_0$ \'etale) and \cite[Thm.~6]{elkik} ($f_0$ smooth).  
For non-affine schemes and algebraic spaces, these results are new and answer
positively a question of
Temkin and Tyomkin~\cite[Qstn.~5.3]{temkin-tyomkin_Ferrand}. 

\subsection*{Application: Local structure of stacks} We now generalize \Cref{T:affine-etale-nbhds} from extending affine \'etale neighborhoods to extending linearly fundamental \'etale neighborhoods.  By definition, an algebraic stack $\stX$ is \emph{fundamental} if there is an affine morphism $\stX \to B\GL_{n, \ZZ}$ for some $n$, and  \emph{linearly fundamental} if it is fundamental and cohomologically affine; see \cite[\S2.2]{alper-hall-rydh_etale-local-stacks} for further discussion.

In order to formulate mixed-characteristic versions of the local structure results, we recall from \cite[\S15]{alper-hall-rydh_etale-local-stacks} the following conditions on an algebraic stack $\stX$.
\begin{enumerate}
  \myitem[FC]\label{Cond:FC} There is only a finite number of
    different characteristics in $\stX$.
  \myitem[PC]\label{Cond:PC} Every closed point of $\stX$ has
    positive characteristic.
  \myitem[N]\label{Cond:N} Every closed point of $\stX$ has nice
    stabilizer \cite[Defn.~1.1]{hall-rydh_alg-groups-classifying} (i.e., is an extension of a finite linearly reductive group scheme by an algebraic group of multiplicative type).
  \end{enumerate}
If $\stX$ is linearly fundamental, then \ref{Cond:PC}$\implies$\ref{Cond:N} as linearly reductive group schemes in positive characteristic are nice \cite{nagata_complete-reducibility}, \cite[Thm.~1.2]{hall-rydh_alg-groups-classifying}.
The condition that we often impose will be of the following
form for some morphism of stacks $\stW_0\to \stX$: assume either that $\stW_0$ satisfies \ref{Cond:N}, or $\stX$ satisfies \ref{Cond:FC}.

We also remind the reader of another type of algebraic stack from  \cite[\S2.2]{alper-hall-rydh_etale-local-stacks}: an algebraic stack $\stX$ is \emph{nicely fundamental} if it admits an affine morphism to $B_SQ$, where $Q \to S$ is a nice and embeddable group scheme over $S$. It follows that nicely fundamental stacks are linearly fundamental. 

\begin{theorem}[Local structure of stacks]\label{T:local-structure:excellent}
Let $S$ be an excellent algebraic space and let $\stX$ be an algebraic
stack, quasi-separated and locally of finite presentation over $S$ with
affine stabilizer groups.  Consider a diagram
\[\xymatrix{
  \stW_0 \ar@{^(-->}[r] \ar[d]^{f_0}  & \stW \ar@{-->}[d]^f \\
  \stX_0 \ar@{^(->}[r]  & \stX,   
}\]
where $\stX_0\inj \stX$ is a closed immersion and $f_0\colon \stW_0\to \stX_0$ is 
a morphism of algebraic stacks with $\stW_0$ linearly fundamental.
\begin{enumerate}
\item \label{TI:local-structure:excellent1}
  If $f_0$ is smooth (resp.\ \'etale), then there exists a smooth
  (resp.\ \'etale) morphism $f\colon \stW\to \stX$ such that $\stW$ is
  fundamental and $f|_{\stX_0}\simeq f_0$.
\item \label{TI:local-structure:excellent2}
  Assume that $\stW_0$ satisfies \ref{Cond:PC} or \ref{Cond:N} or $\stX_0$ satisfies
  \ref{Cond:FC}. If $f_0$ is syntomic and $\stX_0$ has the resolution property, then
  there exists a syntomic morphism $f\colon \stW\to \stX$ such that $\stW$ is
  fundamental and $f|_{\stX_0}\cong f_0$.
\end{enumerate}
\end{theorem}

Syntomic means flat and locally of finite presentation, with fibers that are
local complete intersections. An important example in our context is that
any morphism $BG\to \stG_x$ is smooth in characteristic zero but merely
syntomic in positive characteristic.

For further refinements on $\stW$, see
\Cref{T:refinement1,T:refinement2} below and
\cite[\S16--17]{alper-hall-rydh_etale-local-stacks}.
For a non-noetherian version,
see \Cref{T:local-structure:non-noeth}. We also have the following result.

\begin{theorem}[Local structure of stacks at non-closed points]\label{T:local-structure:non-closed}
Let $\stX$ be a quasi-separated algebraic stack with affine
stabilizer groups. Let $x\in |\stX|$
be a point with residual gerbe $\stG_x$ and let $f_0\colon \stW_0\to \stG_x$ be
a syntomic (resp.\ smooth, resp.\ \'etale) morphism with $\stW_0$ linearly fundamental. 
Then there exists a
syntomic (resp.\ smooth, resp.\ \'etale) morphism $f\colon \stW\to \stX$ such
that $\stW$ is fundamental and $f|_{\stG_x}\cong f_0$.
\end{theorem}

We give a more general version for pro-affine-immersions in
\Cref{T:local-structure:non-noeth:pro-immersion}.  Note that the inclusions
$\stX_0 \inj \stX$ of a closed substack in \Cref{T:local-structure:excellent}
and $\stG_x \inj \stX$ of a residual gerbe in
\Cref{T:local-structure:non-closed} are both pro-affine-immersions. We also
have refinements on the local charts (cf.\ \cite[Prop.\ 12.5 and
  Cor.~17.4]{alper-hall-rydh_etale-local-stacks}).

\begin{theorem}[Refinement 1]\label{T:refinement1}
Let $\stW$ be a fundamental stack. Let $\stW_0\inj \stW$ be a pro-affine-immersion. Assume that $\stW_0$ is linearly fundamental and satisfies \ref{Cond:PC}, \ref{Cond:N}, or
 \ref{Cond:FC}. If
$g\colon \stW\to \stX$ is a morphism to an algebraic stack with affine (resp.\
separated) diagonal,
such that $g|_{\stW_0}$ is representable, then there exists an \'etale neighborhood
$\stW'\to \stW$ of $\stW_0$ such that $\stW'$ is fundamental and
$g|_{\stW'}$ is affine (resp.\ representable).
\end{theorem}
\begin{theorem}[Refinement 2]\label{T:refinement2}
Let $\stW$ be a fundamental stack and $\stW_0\inj \stW$ be a
pro-affine-immersion. Assume that $\stW_0$ is linearly fundamental and
that either $\stW_0$ satisfies \ref{Cond:PC}, \ref{Cond:N}, or $\stW$
satisfies \ref{Cond:FC}. Then there exists an \'etale neighborhood
$\stW'\to \stW$
of $\stW_0$ such that
\begin{enumerate}
\item $\stW'$ is linearly fundamental.
\item If $\stW_0=[\Spec A_0/G_0]$, where $G_0$ is a linearly
  reductive (resp.\ nice)
    and embeddable group scheme over the good moduli
  space $W_0$, then we can arrange so that $\stW'=[\Spec A/G]$, where $G$ is a
  linearly reductive (resp.\ nice) and embeddable group scheme over the good
  moduli space $W'$, such that $G|_{W_0} \cong G_0$.
\end{enumerate}
\end{theorem}
\subsection*{Application: Henselizations}

The {\it henselization} of an algebraic stack $\stX$ along a morphism $\nu \co \stW \to \stX$ is an initial object
in the 2-category of 2-commutative diagrams
\[
  \xymatrix{
  \stW \ar[r] \ar[rd]_{\nu}   & \stX' \ar[d]^f \\
  & \stX,
}
\]
where $f \co \stX' \to \stX$ is pro-\'etale.  Recall that $f \co \stX' \to \stX$ is called {\it pro-\'etale} if it is an inverse limit of quasi-separated \'etale
neighborhoods $\stX_\lambda\to \stX$ such that the transition maps
$\stX_\lambda\to \stX_\mu$ are affine for all sufficiently large $\lambda\geq
\mu$. Note that we do not require that $f$ is representable or separated.

\begin{theorem}[Existence of Henselizations]\label{T:henselization}
  Let $\stX$ be a quasi-separated algebraic stack with affine stabilizers. Let $\nu \co \stX_0 \inj \stX$ either be the inclusion of a closed substack satisfying \ref{Cond:PC}, \ref{Cond:N}, or
  \ref{Cond:FC}; or the inclusion of a residual gerbe.  If $\stX_0$ is linearly fundamental, then the henselization $\stX_{\nu}^{h}$ of $\stX$ along $\nu$ exists.  Moreover, $\stX_{\nu}^{h}$ is linearly fundamental and $(\stX_{\nu}^{h}, \stX_0)$ is a henselian pair.
\end{theorem}

When $\stX$ is an affine scheme, then \Cref{T:henselization} is
\cite[Ch.~XI, Thm.~2]{raynaud_hensel_rings}.  The result is new for non-affine schemes and algebraic spaces. It is also closely related to, but does not settle, conjectures of Greco and
Strano on henselian schemes~\cite[Conj.~A, B and C]{greco-strano}.

Note that there are no analogous results for open neighborhoods: there are
schemes with affine closed subschemes that do not admit affine
neighborhoods. Indeed, there is a separated scheme with two closed points that
does not admit an affine open neighborhood and such that the semi-localization
at the two points does not exist. See \Cref{A:Zariskification}.

\subsection*{Application: Ferrand pushouts}
As an application of \Cref{T:affine-etale-nbhds} we can prove that
Ferrand pushouts~\cite{ferrand_conducteur,temkin-tyomkin_Ferrand} exist for
algebraic spaces and algebraic stacks. In the affine case, these are
Milnor squares~\cite[\S 2]{milnor_intro-alg-K-theory} and it follows that
these are pushouts in the category of quasi-separated algebraic stacks.

\begin{theorem}[Existence of Ferrand pushouts]\label{T:Ferrand}
Consider a diagram
\[
\xymatrix{%
\stX_0 \ar@{^(->}[r]^{i}\ar[d]_{f} & \stX\\
\stY_0
}%
\]
of quasi-separated algebraic stacks where $i$ is a closed immersion and
$f$ is affine. Then the pushout $\stY$ exists in the category of
quasi-separated algebraic stacks and is a geometric pushout.
If $\stX_0$, $\stY_0$ and $\stX$ are
Deligne--Mumford stacks (resp.\ algebraic spaces, resp.\ affine schemes), then so is $\stY$.
\end{theorem}

\Cref{T:Ferrand} generalizes the main theorem of~\cite{temkin-tyomkin_Ferrand},
where certain pushouts of algebraic spaces are proven to exist.
\subsection*{Application: Nisnevich neighborhoods}

The following application is used in \cite{hoyois-krishna_vanishing} and is a simple consequence of the local structure at non-closed points (\Cref{T:local-structure:non-closed}).

\begin{theorem}[Nisnevich neighborhoods of stacks with nice stabilizers]
\label{T:nisnevich-neighborhoods}
Let $\stX$ be a quasi-compact and quasi-separated algebraic stack such that
every, not necessarily closed, point of $\stX$ has nice stabilizer group. Then
there is a Nisnevich covering $f\colon \stW\to \stX$, where $\stW$ is nicely
fundamental. That is,
\begin{enumerate}
\item $f$ is \'etale and for every, not necessarily closed, point $x\in |\stX|$
  the restriction $f|_{\stG_x}$ has a section.
\item $\stW$ admits an affine good moduli space $W$ and there is a nice
  embeddable group scheme $G\to W$ such that $\stW=[\Spec A/G]$.
\end{enumerate}
If $\stX$ has affine (resp.\ separated) diagonal, then we can arrange
that $f$ is affine (resp.\ representable).
\end{theorem}
\begin{remark}
When $\stX$ is an algebraic stack with a good moduli space such that every
point of characteristic zero has an open neighborhood of characteristic zero,
then $\stX$ has a strong Nisnevich neighborhood of the form $[\Spec A/G]$ with $G$
linearly reductive~\cite[Thm.~13.1]{alper-hall-rydh_etale-local-stacks}.
Here strong means that the Nisnevich neighborhood is a pull-back from a
Nisnevich cover of the good moduli space.
Note that the condition that $\stX$ admits a good moduli space implies that every \emph{closed} point has linearly reductive stabilizer.  
\end{remark}
In the case of linearly reductive stabilizers at \emph{closed} points, we have
the following result.
\begin{theorem}[Nisnevich neighborhoods of stacks with linearly reductive stabilizers at closed points]\label{T:nisnevich-neighborhoods-LR}
  Let $\stX$ be a quasi-compact and quasi-separated algebraic stack
  with affine stabilizers and linearly reductive stabilizers at closed
  points.
  Assume that $\stX$ has separated (resp.\ quasi-affine, resp.\ affine)
  diagonal. Then there is a Nisnevich covering $f\colon [V/\GL_m] \to \stX$,
  where $V$ is a quasi-compact separated algebraic space
  (resp.\ quasi-affine scheme, resp.\ affine scheme).  In general, the
  morphism $f$ is not representable but if $\stX$ has affine diagonal we can
  also arrange so that $f$ is affine.
\end{theorem}
When $\stX$ has affine diagonal, the Nisnevich covering is fundamental but not
always linearly fundamental. If $\stX$ is the stack quotient of the
non-separated affine line by $\ZZ/2\ZZ \times
\mathbb{G}_m$~\cite[Ex.~5.2]{alper-hall-rydh_luna-stacks}, then the unique
closed point has stabilizer $\mathbb{G}_m$ whereas the open point has
stabilizer $\ZZ/2\ZZ$. Every Nisnevich covering will thus have a point with
stabilizer $\ZZ/2\ZZ$ and such stacks are not linearly fundamental in
characteristic $2$.

\subsection*{Application: Compact generation}
Let $\stX$ be a quasi-compact and quasi-separated algebraic stack and
consider its unbounded derived category of $\sO_\stX$-modules with
quasi-coherent cohomology sheaves $\DQCOH(\stX)$. A vexing question over
the years has been whether the category $\DQCOH(\stX)$ is \emph{compactly
  generated}. In this situation, this is equivalent to finding a set
of perfect complexes $\{P_\lambda\}_{\lambda\in \Lambda}$ on $\stX$ such that
\begin{enumerate}[label=(\alph*)]
\item if $M \in \DQCOH(\stX)$ and $\Hom_{\sO_\stX}(P_\lambda,M) = 0$ for all
$\lambda \in \Lambda$, then $M = 0$; and
\item the functor
  $\Hom_\stX(P_\lambda,-) \colon \DQCOH(\stX) \to \mathsf{Ab}$ preserves
  small coproducts for all $\lambda\in \Lambda$.
\end{enumerate}
For schemes, definitive positive results go back to the pioneering work of
\cite{MR1106918,MR1308405}. For a thorough discussion on the
subtleties of this question for algebraic stacks, we refer the
interested reader to
\cite{hall-rydh_perfect-complexes,hall-neeman-rydh_no-compacts}.

A lot of progress was made on this question for stacks in \cite[Thm.\
5.1]{alper-hall-rydh_luna-stacks} and \cite[Prop.\
14.1]{alper-hall-rydh_etale-local-stacks}, however. More precisely, \cite[Thm.\
5.1]{alper-hall-rydh_luna-stacks} established compact generation
provided that $\stX$ had affine diagonal and the identity component
$G_x^0$ of the stabilizer groups $G_x$ of $\stX$ at all points $x$ of $\stX$
were of multiplicative type. It was shown in \cite[Thm.\
1.1]{hall-neeman-rydh_no-compacts}, however, that if $\stX$ had a point of positive characteristic
$y$ such that the reduced identity component $(G_y)_{\mathrm{red}}^0$ was not a torus, then
$\DQCOH(\stX)$ was not compactly generated. In the following theorem we eliminate 
this discrepancy and give the following characterization of
algebraic stacks in positive characteristic that have compactly
generated derived categories.
\begin{theorem}\label{T:compact-generation-derived-category}
  Let $\stX$ be a quasi-compact algebraic stack with
  affine diagonal satisfying \ref{Cond:PC}. The following conditions
  are equivalent.
  \begin{enumerate}
  \item\label{TI:compact-generation-derived-category:crisp} $\stX$ is
    $\aleph_0$-crisp \cite[Defn.\ 8.1]{hall-rydh_perfect-complexes}.
  \item\label{TI:compact-generation-derived-category:support}
    $\DQCOH(\stX)$ is compactly generated and for every closed subset
    $Z \subseteq |\stX|$ with quasi-compact complement, there exists a
    perfect complex $P$ on $\stX$ with $\mathrm{supp}(P) = Z$.
  \item\label{TI:compact-generation-derived-category:generated} $\DQCOH(\stX)$ is compactly generated.
  \item\label{TI:compact-generation-derived-category:points} For every point $x$ of $\stX$, the reduced identity
    component $(G_x)_{\mathrm{red}}^0$ of the stabilizer $G_x$ at $x$
    is a torus.
  \item\label{TI:compact-generation-derived-category:closed-points} For every closed point $x$ of $\stX$, the reduced identity
    component $(G_x)_{\mathrm{red}}^0$ of the stabilizer $G_x$ at $x$
    is a torus.
  \end{enumerate}
\end{theorem}
We will prove \Cref{T:compact-generation-derived-category}
immediately after the non-noetherian local structure
\Cref{T:local-structure:non-noeth}, and make use of the
refinements established in \cite{alper-hall-rydh_etale-local-stacks}.

\subsection*{Application: Local structure theorem of derived algebraic stacks}
We now come to the derived versions of our local structure results. Recall that a morphism $f$ of derived stacks is \emph{quasi-smooth}
if $f$ is locally of finite presentation and its cotangent complex $\LL_f$ has Tor-amplitude $\leq 1$.
This is the analogue of lci maps in derived algebraic geometry.

\begin{theorem}[Local structure of derived stacks]\label{T:local-structure:derived}
Let $\stX$ be a quasi-separated algebraic derived $1$-stack with affine stabilizers.
Let $\stX_0\inj \stX$ be a
closed substack and let $f_0\colon \stW_0\to \stX_0$ be a morphism with
$(\stW_0)_{\cl}$ linearly fundamental.  Assume one of the following conditions:
\begin{enumerate}
\item $\stW_0$ satisfies \ref{Cond:PC} or \ref{Cond:N}; or
\item $\stX_0$ satisfies \ref{Cond:FC}.
\end{enumerate}
Then
\begin{enumerate}[label=(\alph*)]
\item If $f_0$ is smooth (resp.\ \'etale), then there exists a smooth
  (resp.\ \'etale) morphism $f\colon \stW\to \stX$ such that $\stW$ is
  fundamental and $f|_{\stX_0}\cong f_0$.
\item Assume that $(\stX_0)_\cl$ has the resolution property. If $f_0$ is
  quasi-smooth then there exists a
  quasi-smooth morphism $f\colon \stW\to \stX$ such that $\stW$ is fundamental
  and $f|_{\stX_0}\cong f_0$ (here the restriction 
  denotes the derived pull-back).
\end{enumerate}
\end{theorem}
It follows from \Cref{P:derived-effectivity} that $\stX$ is
linearly fundamental if and only if the underlying classical stack $\stX_{\cl}$ is linearly fundamental. See \Cref{S:local-structure-derived} for further discussion. 
\subsection*{Application: Local structure of a \texorpdfstring{$\Theta$}{Theta}-stratum}
Let $\mathscr{S}$ be a quasi-separated algebraic stack and let $\stX$ be an
algebraic stack, quasi-separated and locally of finite presentation over
$\mathscr{S}$ with affine stabilizers relative to $\mathscr{S}$. Let $\Theta :=
[\mathbb{A}^1 / \mathbb{G}_m]$; then the mapping stack
$\mathrm{Filt}(\stX):=\underline{\Map}_\mathscr{S}(\Theta_{\mathscr{S}},\stX)$
is also algebraic, locally of finite presentation, quasi-separated, and has
affine stabilizers relative to $\mathscr{S}$
\cite[Prop.~1.1.2]{hlinstability}. A $\Theta$-stratum in $\stX$ is by
definition an open and closed substack $\stY \subset \mathrm{Filt}(\stX)$ such
that the morphism $\stY \to \stX$ defined by restricting to $1 \in \Theta$ is a
closed immersion, so that we may also regard $\stY$ as a closed substack of
$\stX$ (see \cite[Defn.~2.1.1]{hlinstability}).

Stratifications by closed substacks of this kind arise in geometric invariant theory, as well as on moduli stacks such as the moduli of torsion free sheaves on a projective scheme. In \cite[Lem.~6.11]{AHLH}, the following local structure result was established using our \Cref{T:local-structure:non-noeth}, and it is key to proving the semistable reduction theorem \cite[Thm.~6.3]{AHLH}.
\begin{proposition}
Let $S$ be a noetherian algebraic space. Let $\stX$ be an algebraic stack
of finite type over $S$ with affine diagonal over $S$. If $\stY \hookrightarrow \stX$ is a $\Theta$-stratum, then there is a smooth representable morphism $p \colon [\Spec(A)/\mathbb{G}_m] \to \stX$ such that $\stY$ is contained in the image of $p$, and $p^{-1}(\stY)$ is the $\Theta$-stratum
\[
p^{-1}(\stY) = [\Spec(A/I_+)/\mathbb{G}_m] \hookrightarrow [\Spec(A)/\mathbb{G}_m],
\]
where $I_+ \subset A$ is the ideal generated by elements of positive degree.
\end{proposition}
\end{section}


\begin{section}{Artin algebraization}\label{S:algebraization}

In this section we prove Artin's algebraization theorem for 
linearly fundamental pairs (\Cref{T:Artin-algebraization}) which establishes 
\Cref{T:artin-algebraization-intro} as a special case. In order to state the theorem, we will need the following terminology.

\begin{definition}
A \emph{pair} $(\stX,\stX_0)$ consists of an algebraic stack $\stX$ and a
closed substack $\stX_0$. We let $\sI_\stX$ denote the ideal defining $\stX_0$
and let $\stX_n$ denote the $n$th infinitesimal neighborhood of
$\stX_0$, that is, the closed substack defined by $\sI_\stX^{n+1}$.  We say 
that a pair $(\stX, \stX_0)$ has a given property $\sP$ (e.g. linearly fundamental) 
if both $\stX$ and $\stX_0$ have $\sP$.

A \emph{morphism of pairs} $(\stX,\stX_0)\to (\stY,\stY_0)$ is a morphism
$f\co \stX\to \stY$ such that $\stX_0\inj f^{-1}(\stY_0)$, or
equivalently, $f^{-1}\sI_\stY\subseteq \sI_\stX$. For any $n\geq 0$, we
let $f_n\co \stX_n\to \stY_n$ denote the induced morphism.
We say that $f$ is \emph{adic} if $\stX_0=f^{-1}(\stY_0)$.
\end{definition}

Note that if $f$ is adic, then $\stX_n=f^{-1}(\stY_n)$ for all $n$.

\begin{definition} \label{D:formally_versal}
  Let $f\colon \stZ \to \stX$  be a morphism of functors or stacks (e.g., schemes or
  algebraic spaces). Let $T$ be a stack and $T\to \stZ$ a morphism.  We say that
  \emph{$f$ is formally versal at $T$} if the following condition holds:
  For every nilpotent immersions $T\inj T'\inj T''$ and $2$-commutative diagram
  of solid arrows
  \[
  \xymatrix{
  T\ar[r] & T'\ar[d]\ar[r] & \stZ \ar[d]^f\\
          & T''\ar[r]\ar@{-->}[ur] & \stX,
  }%
  \]
  there exists a lift $T''\to \stZ$ and $2$-morphisms that make the whole diagram
  $2$-commutative.
  \end{definition}

Our main theorem is the following result, which generalizes
\cite[Cor.~A.19]{alper-hall-rydh_luna-stacks} and
\cite[Thm.~12.14]{alper-hall-rydh_etale-local-stacks}.

\begin{theorem}[Algebraization of linearly fundamental pairs]\label{T:Artin-algebraization}
Let $S$ be an excellent affine scheme. Let $\stX$ be an algebraic stack, locally of finite type over $S$ with quasi-separated diagonal.
Let $(\stZ,\stZ_0)$ be a complete linearly fundamental pair
(\Cref{D:complete-pair}) over $S$ such that $\stZ_0$ is of finite
type over $S$. Let $\eta\colon \stZ\to \stX$ be a morphism, formally versal at
$\stZ_0$. Then there exists
\begin{enumerate}
\item a fundamental pair $(\stW,\stW_0)$ such that
  $\stW\to S$ is of finite type and $\stW_0$ is linearly fundamental;
\item a morphism $\varphi\colon (\stZ,\stZ_0)\to (\stW,\stW_0)$
  such that $\varphi_n\colon \stZ_n \to \stW_n$ is an isomorphism for all $n \geq 0$.
\item a $2$-commutative diagram over $S$
\[
\xymatrix{\stZ \ar[r]^\varphi \ar@/_2ex/[rr]_{\eta} & \stW \ar[r]^{\xi} & \stX}
\]
\end{enumerate}
In particular, the induced map $\widehat{\varphi}\colon \stZ \to \widehat{\stW}$ is an isomorphism and $\xi$ is smooth in a neighborhood of $\stW_0$.
\end{theorem}

\begin{remark} \label{R:artin-algebraization}
Most of the statement of the theorem remains valid, with the same proof, when $\stX$ is an arbitrary category fibered in groupoids that is locally of finite presentation over $S$. The only difference is that instead of a $2$-isomorphism $\xi \circ \varphi \simeq \eta$, one only obtains a compatible family of $2$-isomorphisms $\xi \circ \varphi|_{\stZ_n} \simeq \eta|_{\stZ_n}$ for all $n\geq 0$.
\end{remark}

We prove this theorem at the end of the section after discussing some background
material on pairs.  We first explain how this theorem implies \Cref{T:artin-algebraization-intro}.

\begin{proof}[Proof of \Cref{T:artin-algebraization-intro}]
  Applying \Cref{T:Artin-algebraization} and \Cref{R:artin-algebraization} 
  with $(\stZ, \stZ_0) := (Z, Z_0)$ gives a fundamental pair $(\stW, \stW_0)$ with $\stW_0 \cong Z_0$.  Since $Z_0$ is affine, we may apply \cite[Prop.~12.5]{alper-hall-rydh_etale-local-stacks} to the morphism $\stW \to S$ to conclude that there is affine open neighborhood $U \subset \stW$ of $Z_0$. Replacing $(\stW, \stW_0)$ with $(U, Z_0)$ gives the result. 
\end{proof}

\subsection{Coherently complete pairs}
The following definition was introduced in \cite{alper-hall-rydh_luna-stacks} and 
was further studied in \cite{alper-hall-rydh_etale-local-stacks}.
\begin{definition}\label{D:complete-pair}
  We say that a pair $(\stX,\stX_0)$ is \emph{complete}, or that $\stX$ is
  \emph{coherently complete along $\stX_0$}, if $\stX$ is noetherian with affine
  diagonal and the induced functor $\Coh(\stX)\to \varprojlim_n \Coh(\stX_n)$ is an equivalence of
  abelian categories of coherent sheaves.
\end{definition}

  By Tannaka duality \cite{hall-rydh_coherent-tannaka-duality}, we have that
  $\stX$ is the colimit of $\{\stX_n\}_{n\geq 0}$ in the category of noetherian
  stacks with quasi-affine diagonal and also in the category of noetherian stacks
  with affine stabilizers if $\stX_0$ is quasi-excellent.
  
  Let $(\stX,\stX_0)$ be a linearly fundamental noetherian pair.  The good moduli
  space $X$ is a noetherian affine scheme and $\pi\co \stX\to X$ is of finite
  type. This gives a morphism of pairs $(\stX,\stX_0)\to (X,X_0)$ where
  $X_0=\pi(\stX_0)$.  The pair $(\stX,\stX_0)$ is complete if and only if
  $(X,X_0)$ is complete \cite[Thm.~1.6]{alper-hall-rydh_etale-local-stacks}. The
  latter simply means that if $X=\Spec A$ and $X_0=\Spec A/I$, then $A$ is
  $I$-adically complete.
  
  If $(\stX,\stX_0)$ is a fundamental noetherian pair
  such that $\stX_0$ is linearly fundamental, then $(\widehat{\stX},\stX_0)$ is a
  complete linearly fundamental pair where
  $\widehat{\stX}=\stX\times_X \widehat{X}$ and $\widehat{X}=\Spec \widehat{A}$
  is the $I$-adic completion. Indeed, the completion factors through
  the Zariskification $\stX\times_X \Spec \bigl( (1+I)^{-1}A \bigl)$, which is linearly fundamental by~\cite[Cor.~13.7]{alper-hall-rydh_etale-local-stacks}.

\subsection{Preliminary results on pairs}
In this section, we provide criteria to check that a 
morphism of pairs is a closed immersion or isomorphism (\Cref{P:closed/iso-cond:complete}) or is formally versal (\Cref{L:formally_versal}).
\begin{lemma}{\cite[Prop.~1.2]{vasconcelos_fin-gen-flat-modules}}\label{L:vasconcelos}
Let $A$ be a ring and let $\phi\colon M\to N$ be a surjective homomorphism of
finitely generated $A$-modules. If there exists an $A$-module isomorphism
$M\cong N$, then $\phi$ is an isomorphism.
\end{lemma}
\begin{proof}
We identify $N$ with $M$ and treat $\phi$ as an endomorphism of $M$. Then
$M$ is also a module over $A[t]$ where $tx=\phi(x)$ for $x\in M$. Since $\phi$
is surjective $tM=M$ and Nakayama's lemma tells us that there is an element
$a\in A[t]$ such that $(1-at)M=0$. That is $\phi$ has inverse given by
$\phi^{-1}(x)=ax$.
\end{proof}

\begin{lemma}\label{L:ring-iso:nilpotent}
Suppose that $I\subseteq R$ is an ideal and $\varphi\colon R\to S$ is a
surjective homomorphism of noetherian rings. If there is an abstract
isomorphism of graded
$R/I$-modules $\Gr_I R\to \Gr_{IS} S$ and $I$ is nilpotent, then
$\varphi$ is an isomorphism.
\end{lemma}
\begin{proof}
Since $\varphi$ is surjective, it induces a surjection $\Gr_n \varphi\colon
I^n/I^{n+1}\to I^nS/I^{n+1}S$ of finitely generated $R/I$-modules. By
assumption, there is an abstract isomorphism $I^n/I^{n+1}\to I^nS/I^{n+1}S$
of $R/I$-modules,
so $\Gr_n \varphi$ is an isomorphism by \Cref{L:vasconcelos}.

We have induced morphisms of exact sequences
\[
\xymatrix{%
0\ar[r] & I^d/I^{d+1}\ar[r]\ar[d]^{\Gr_d\varphi} & R/I^{d+1}\ar[r]\ar[d]^{\varphi_{d+1}} & R/I^d\ar[r]\ar[d]^{\varphi_d} & 0\\
0\ar[r] & I^d S/I^{d+1} S\ar[r] & S/I^{d+1} S\ar[r] & S/I^d S\ar[r] & 0
}
\]
and it follows that $\varphi_d\colon R/I^d\to S/I^dS$ is an isomorphism for
every $d\geq 0$ by induction on $d$. Since $I^n=0$ for sufficiently large $n$,
the result follows.
\end{proof}

The following results generalize~\cite[Props.~A.8 and A.10]{alper-hall-rydh_luna-stacks} from the local case.

\begin{proposition}\label{P:closed/iso-cond:nilpotent}
Let $f\colon (\stX,\stX_0)\to (\stY,\stY_0)$ be a morphism of noetherian pairs.
\begin{enumerate}
\item\label{PI:closed:nilpotent}
  If $f_1$ is a closed immersion, then so is $f_n$ for every $n\geq 0$.
\item\label{PI:adic:nilpotent}
  If $f_1$ is a closed immersion and $f_0$ is an isomorphism, then $f_n$ is
  adic for every $n\geq 0$.
\item\label{PI:iso:nilpotent}
  If $f_1$ is a closed immersion and there
  exists an isomorphism of graded $\sO_{\stY_0}$-modules $\psi\colon \Gr_{\sI_\stY}(\sO_\stY)\to
  (f_0)_*\Gr_{\sI_\stX}(\sO_\stX)$, then
  $f_n$ is an isomorphism for every $n\geq 0$.
\end{enumerate}
\end{proposition}
\begin{proof}
We can replace $f$ with $f_n$.
The first part is then~\cite[Lem.~6.3]{alper-hall-rydh_etale-local-stacks}: the
question is local and reduces to the affine case where it follows from
Nakayama's lemma. For the second part, we have seen that $f_n$ is a closed
immersion and then it is adic if and only if $f_0$ is an isomorphism.
The third part is also local and thus follows from
\Cref{L:ring-iso:nilpotent}.
\end{proof}

\begin{proposition}\label{P:closed/iso-cond:complete}
Let $f\colon (\stX,\stX_0)\to (\stY,\stY_0)$ be a morphism of complete pairs
such that $f_0$ is an isomorphism.
\begin{enumerate}
\item\label{PI:closed}
  $f$ is a closed immersion if and only if $f_1$
  is a closed immersion.
\item\label{PI:iso}
  $f$ is an isomorphism if and only if $f_1$ is a closed immersion and there
  exists an isomorphism $\psi\colon \Gr_{\sI_\stY}(\sO_\stY)
  \to (f_0)_*\Gr_{\sI_\stX}(\sO_\stX)$ of graded $\sO_{\stY_0}$-modules.
\end{enumerate}
\end{proposition}
\begin{proof}
The conditions are clearly necessary. Conversely, if the conditions of
\itemref{PI:closed} (resp.\ \itemref{PI:iso}) hold, then $f_n$ is adic and
a closed immersion (resp.\ an isomorphism) for every $n\geq 0$ by
\Cref{P:closed/iso-cond:nilpotent}.  Since $f_n$ is adic, we have that
$f_n^{-1}(\stY_m)=\stX_m$ for all $m\leq n$.
Since $\stY$ is coherently complete along $\stY_0$, we obtain a closed substack
$\stZ\inj \stY$ such that $\stZ\times_\stY \stY_n=\stX_n$ for all $n\geq 0$.
Under condition \itemref{PI:iso}, we have that $\stZ=\stY$. Finally, since
$(\stX,\stX_0)$ is complete, we have by Tannaka duality a unique isomorphism
$\stX\to \stZ$ over $\stY$.
\end{proof}

Let $X$ be a quasi-compact and quasi-separated algebraic stack. Recall
\cite[Defn.\ 2.1]{hall-rydh_perfect-complexes} that $X$ is said to
have \emph{cohomological dimension $0$} if $\mathrm{H}^i(X,M) = 0$ for
all $i>0$ and quasi-coherent $\sO_X$-modules $M$. Affine schemes have
cohomological dimension $0$. More generally, cohomologically affine
algebraic stacks that have affine diagonal or are noetherian and
affine-pointed also have cohomological dimension $0$ \cite[Thm.\
C.1]{hall-neeman-rydh_no-compacts}.

\begin{lemma}\label{L:formally_versal}
Let $f\colon (\stZ,\stZ_0) \to (\stX,\stX_0)$ be a morphism of locally noetherian pairs. If $f_n\co \stZ_n \to \stX_n$ is smooth for all $n\geq 0$, then $f$ is formally versal at any morphism $T \to \stZ$ from a quasi-compact and quasi-separated algebraic stack $T$ of cohomological dimension $0$ whose set theoretic image is contained in $|\stZ_0|$.
\end{lemma}
\begin{proof}
The lifting criterion in \Cref{D:formally_versal} is equivalent to the same
criterion for the map $f_n\co \stZ_n \to \stX_n$ for $n \gg 0$ large enough
that $\stZ_n$ contains the image of $T'$ and $\stX_n$ contains the image of
$T''$, so by our hypotheses we may assume that the map $f$ is smooth. First
note that $T'$ has cohomological dimension $0$ because any quasi-coherent
$\sO_{T'}$-module admits a finite filtration whose associated graded objects
are pushforwards of objects in $\QCoh(T)$. Also, because we may factor $T' \to
T''$ into a sequence of square-zero extensions, it suffices to verify the
lifting criterion in the case where $T' \to T''$ is a square-zero extension by
some $M \in \QCoh(T')$. In this case the obstruction to the existence of a
dotted arrow is an element in the group
$\Ext^1_{T'}(\LL_{\stZ/\stX}|_{T'},M)$. Since $f$ is smooth, $\LL_{\stZ/\stX}$
is a perfect complex of Tor-amplitude $[0,1]$. Hence, the $\Ext$ group vanishes
as $T'$ has cohomological dimension~$0$.
\end{proof}

\subsection{Proof of \texorpdfstring{\Cref{T:Artin-algebraization}}{Theorem~\ref{T:Artin-algebraization}}}
First we establish an important special case of Artin algebraization for pairs:

\begin{lemma} \label{L:approx_basic_case}
Let $(S,S_0)$ be an excellent affine pair, let $(T,T_0)$ be a complete affine
pair, and let $f\colon (T,T_0)\to (S,S_0)$ be a morphism such that $f_0$ is an
isomorphism and $f_1$ is a closed immersion. Let $\stX$ be a finite type
algebraic stack over $S$, and let $\stZ_0 \hookrightarrow \stZ:=T \times_S
\stX$ be a closed substack over $T_0$. For any $N\geq 0$, there is an affine
\'{e}tale neighborhood $(S',S'_0)\to (S,S_0)$ and a closed substack $\stW \inj
S' \times_S \stX$ such that:
\begin{enumerate}
\item The map $T \to S$ factors through $S'$, and $T_N \to S'_N$ is a closed immersion;
\item\label{TI:approx-last} $T_N \times_T \stZ = S'_N \times_{S'} \stW$ as closed substacks of $S'_N \times_S \stX$. In particular, if $\stW_0 := \stZ_0 \inj \stW$, then the canonical map is an isomorphism $\stZ_N \cong \stW_N$; and
\item\label{TI:approx-graded-iso} There is an isomorphism $\Gr_{\sI_\stZ} \sO_{\stZ} \cong \Gr_{\sI_\stW} \sO_{\stW}$ of graded modules over $\stZ_0 \cong \stW_0$.
\end{enumerate}
\end{lemma}

\begin{proof}
Consider the functor $F\colon \Sch_{/S}^\op\to \Set$, where $F(U\to S)$ is the set of isomorphism classes of complexes of finitely presented quasi-coherent $\sO_{U \times_S \stX}$-modules $\sE_2\to \sE_1\to \sO_{U \times_S \stX}$ such that $\sE_1$ is locally free. This functor is locally of finite presentation.

Let $\widehat{S}$ be the completion of $S$ along $S_0$. Then $T \to
\widehat{S}$ is a closed immersion by \Cref{P:closed/iso-cond:complete},
because $(T,T_0)$ is complete, $f_0$ is an isomorphism
and $f_1$ is a closed immersion. Now let
\[
\sO_{\widehat{S}}^{\oplus n} \to \sO_{\widehat{S}} \to \sO_{T}
\]
be a presentation of the structure sheaf of $T \inj \widehat{S}$. Pulling back to $\widehat{S} \times_S \stX$ we get a resolution
\[
\ker(\alpha) \xrightarrow{\alpha} \sO_{\widehat{S} \times_S \stX}^{\oplus n} \xrightarrow{\beta} \sO_{\widehat{S} \times_S \stX} \surj \sO_{T \times_S \stX}.
\]
We regard the pair $(\alpha,\beta)$ as an element of $F(\widehat{S})$. Note that by increasing $N$ if necessary, we may assume that both $\alpha$ and $\beta$ satisfy the Artin--Rees condition $\AR{N}$ of \cite[Def.~A.15]{alper-hall-rydh_luna-stacks} with respect to $\stZ_0$.

Let $(S^h,S_0)$ denote the henselization of the pair $(S,S_0)$. By Artin approximation over henselian pairs \cite[Thm.~3.4]{alper-hall-rydh_etale-local-stacks} one can find a class in $F(S^h)$ which restricts to the same class as $(\alpha,\beta)$ in $F(S_N)$. Then because $S^h$ is constructed as an inverse limit of \'etale neighborhoods of $S_0$, we lift this class in $F(S^h)$ to a class $(\alpha',\beta') \in F(S')$ for some \'etale map $S' \to S$ lying under $S^h$ such that $S' \times_S S_0 \simeq S_0$.

We now let $\stW \inj S' \times_S \stX$ be the closed substack defined by $\im(\beta') \subset \sO_{S' \times_S \stX}$. By construction we have
\[
\sO_{\stW} \otimes_{\sO_{S' \times_S \stX}} \sO_{S'_N} \simeq \coker(\beta'|_{S'_N}) \simeq \sO_{T_N \times_S \stX}
\]
as $\sO_{S' \times_S \stX}$-algebras, which is the second condition of the lemma.

Now consider $(\alpha,\beta) \in F(\widehat{S})$ and the restriction of $(\alpha',\beta')$ to $F(\widehat{S})$. Both complexes are isomorphic after tensoring with $\sO_{S_N}$, and by hypothesis the complex defined by $(\alpha,\beta)$ is exact and satisfies the Artin--Rees criterion $\AR{N}$, so the refined Artin--Rees theorem \cite[Thm.~A.16]{alper-hall-rydh_luna-stacks} implies that
\[
\Gr_{\sI_\stZ} \sO_{\stZ} \cong \Gr_{\sI_\stZ} (\coker(\beta)) \cong \Gr_{\sI_\stW} ( \coker(\beta')) \cong \Gr_{\sI_\stW} \sO_{\stW}.\qedhere
\]
\end{proof}

The following generalizes~\cite[Thm.~A.17]{alper-hall-rydh_luna-stacks}.

\begin{proposition}[Weak Artin algebraization for pairs]\label{P:weak-algebraization}
Let $S$ be an excellent affine scheme, and let $\stX$ be a category fibered in groupoids, locally of finite presentation over $S$. Let $(T,T_0)$ be a
noetherian affine pair over $S$ such that $T_0\to S$ is of finite type.
Let $(\stZ,\stZ_0)\to (T,T_0)$ be a morphism of finite presentation and let
$\eta\co \stZ\to \stX$ be a morphism compatible over $S$.
Fix an integer $N \geq 0$. Then there exists
\begin{enumerate}
\item a pair $(\stW,\stW_0)$ of finite presentation over $S$, together
with a morphism $\xi\co \stW\to \stX$;
\item an isomorphism $\stZ_N \cong \stW_N$ over $\stX$; and
\item an isomorphism $\Gr_{\sI_\stZ}\sO_{\stZ}\cong \Gr_{\sI_\stW}\sO_{\stW}$
  of graded modules over $\stZ_0\cong \stW_0$.
\end{enumerate}
Moreover, if $\stZ$ is fundamental, then one can arrange that $\stW$ is fundamental.
\end{proposition}

\begin{proof}
It suffices to prove the claims after base change to the completion of $T$, so we may assume that $T$ is complete along $T_0$. Now write $$T = \varprojlim_\lambda T_\lambda,$$ where $T_\lambda$ is a cofiltered system of affine $S$-schemes of finite type. For $\lambda$ sufficiently large, $T_1 \to T_\lambda$ is a closed immersion. Increasing $\lambda$ if necessary, standard limit methods give us an algebraic stack $\stZ_\lambda$ of finite presentation over $T_\lambda$ fitting into a commutative diagram
\begin{equation} \label{E:approximation_reduction}
\vcenter{\xymatrix{%
\stZ \ar[r]\ar[d] \ar@/^2ex/[rr] & \stZ_\lambda \ar[r]\ar[d] & \stX\ar[d] \\
T\ar[r] & T_\lambda \ar[r]\ar@{}[ul]|\square & S%
}}%
\end{equation}
It now suffices to replace $S$ with $T_\lambda$, and $\stX$ with $\stZ_\lambda$, and to find a stack over $\stZ_\lambda$ meeting the conditions of the theorem. We may therefore assume that $\stX$ is algebraic and of finite presentation over $S$, and that $T_1 \to S$ is a closed immersion, in which case the theorem follows immediately from  \Cref{L:approx_basic_case} with $S_0$ as the
image of $T_0$.

Finally, if $\stZ$ were fundamental, meaning $\stZ$ admits and affine map $f\colon \stZ \to B\GL_{n,\ZZ}$ for some $n$, then in this case one can simultaneously approximate both the map $f$ and the map $\stZ \to \stX$ by replacing $\stX$ with $\stX \times_S (B\GL_{n,S})$ in the argument above. The map $\stZ \to B \GL_{n,S}$ is affine, so \cite[Thm.~C]{rydh_noetherian-approx} guarantees that we can arrange for $\stZ_\lambda$ in \eqref{E:approximation_reduction} to be affine over $B \GL_{n,S}$ as well. The stack $\stW$ constructed in \Cref{L:approx_basic_case} will be affine over $B\GL_{n,S}$ as well, hence fundamental.
\end{proof}

We now prove our main algebraization theorem:

\begin{proof}[Proof of \Cref{T:Artin-algebraization}]
Let $T$ be the good moduli space of $\stZ$ and $T_0$ the good moduli space of $\stZ_0$. Choose an $N \geq 1$. Then $T_0\to S$ is of finite type, so \Cref{P:weak-algebraization} produces a stack $\stW$ satisfying the first two conditions of the theorem along with a map $\xi\colon \stW \to \stX$ and an isomorphism $\psi_N\colon \stW_N \cong \stZ_N$ over $\stX$. 

We would like to extend the isomorphism $\psi_N$ to a compatible sequence of isomorphisms $\psi_n\colon \stW_n \to \stZ_n$ over $\stX$ for all $n\geq N$. Extending the map $\psi_n$ to $\psi_{n+1}$ is equivalent to finding a dotted arrow such that the diagram
\[
\xymatrix@+2ex{ \stW_n\ar[d] \ar[r]^{\psi_n} & \stZ \ar[d]^{\eta} \\ \stW_{n+1} \ar@{-->}[ur]^{\psi_{n+1}} \ar[r]_-{\xi|_{\stW_{n+1}}} & \stX }
\]
is $2$-commutative. It is possible to do this for all $n \geq N$ because by hypothesis the map $\eta$ is formally versal at $\stW_0 = \stZ_0$ (See \Cref{D:formally_versal}). The resulting sequence of maps $\psi_n\colon \stW_n \to \stZ_n$ and the induced map $\widehat{\psi}\colon \widehat{\stW} \to \stZ$ are isomorphisms by \Cref{P:closed/iso-cond:complete} and part (3) of \Cref{P:weak-algebraization}. If we define $\varphi$ to be the inverse of $\psi$ followed by the canonical map $\widehat{\stW} \to \stW$, then by construction we have a compatible sequence of $2$-isomorphisms $\xi \circ \varphi |_{\stZ_{n}} \cong \eta|_{\stZ_{n}}$ for all $n\geq 1$.

If $\stX$ is an algebraic stack with quasi-separated diagonal, then the stack $I := \Isom_{\stZ} (\xi \circ \varphi, \eta)$ is a quasi-separated algebraic space, locally of finite type over $\stZ$. The $2$-isomorphisms $\xi \circ \varphi |_{\stZ_n} \cong \eta|_{\stZ_n}$ give a compatible sequence of sections $\sigma_n$ of $I \to \stZ$ over $\stZ_n$ for all $n \geq 1$. The image of all of the $\sigma_n$ lie in some quasi-compact open substack $I' \subset I$, so we may replace $I$ with $I'$. Then Tannaka duality implies that there is a unique section $\sigma\colon \stZ \to I' \subset I$ of $I \to \stZ$, which corresponds to a $2$-isomorphism $\xi \circ \varphi \simeq \eta$ satisfying the conditions of the theorem.
\end{proof}

\end{section}


\begin{section}{Affine \'etale neighborhoods}
In this section we prove the existence of affine \'etale neighborhoods (\Cref{T:affine-etale-nbhds}). 
\begin{proof}[Proof of \Cref{T:affine-etale-nbhds}] \quad 

\smallskip
\noindent {\bf Step 1: Reduction to \texorpdfstring{$\stX$}{X} of finite presentation over \texorpdfstring{$\ZZ$}{Z}.} 
We may replace $\stX$ with an open
quasi-compact neighborhood of the image of $W_0$. Then $\stX$ is quasi-compact and
quasi-separated and hence of approximation
type~\cite{rydh_noetherian-general}\footnote{When $f_0$ is \'etale, we do not
need \cite{rydh_noetherian-general}. Indeed, then $\stX_0$ is Deligne--Mumford
so
after replacing $\stX$ with an open neighborhood of $\stX_0$, we may assume that
$\stX$ is Deligne--Mumford, hence of global type
and approximation type~\cite[Def.~2.1, Prop.~2.10]{rydh_noetherian-approx}.}.

We can thus write $\stX_0$ as the intersection of finitely presented closed
immersions $\stX_\lambda\inj \stX$~\cite[Thm.~D]{rydh_noetherian-approx}. Using
standard limit methods, we can thus, for
sufficiently large $\lambda$, find an \'etale (resp.\ smooth) morphism
$f_\lambda \co W_\lambda\to \stX_\lambda$ that restricts to
$f_0 \co W_0 \to \stX_0$~\cite[App.~B]{rydh_noetherian-approx}.
After replacing $f_0$ with $f_\lambda$ we can thus assume that
$\stX_0\inj \stX$ is of finite presentation.

Using~\cite[Thm.~D]{rydh_noetherian-approx} we can now write $\stX$ as an
inverse limit of stacks of finite presentation over $\Spec \ZZ$. Using standard
limit methods, we can thus arrange so that the \'etale (resp.\ smooth) map
$f_0\co W_0\to \stX_0$ and the closed immersion $\stX_0\inj \stX$ arise as the
pull-backs from stacks of finite presentation over $\Spec
\ZZ$~\cite[App.~B]{rydh_noetherian-approx}.

In the two reduction steps above, we can also arrange so that $W_0$ remains
affine by~\cite[Thm.~C]{rydh_noetherian-approx}. We can thus assume that $\stX$
is of finite presentation over $\Spec \ZZ$.

\smallskip
\noindent {\bf Step 2: Existence of affine formal neighborhoods.} 
Let $\stX_n$ denote the $n$th infinitesimal neighborhood of $\stX_0$ in $\stX$. We claim that $f_0 \co W_0 \to \stX_0$ lifts to a compatible sequence of cartesian squares
\[
  \xymatrix{
    W_{n-1} \ar[r] \ar[d]_{f_{n-1}} & W_n \ar[d]^{f_n} \\
     \stX_{n-1} \ar[r] & \stX_n
  } 
\]
 such that each $f_n$ is \'etale (resp.\ smooth).  Indeed, by
  \cite[Thm.~1.4]{olsson_def-theory-of-stacks}, the obstruction to lifting
  $f_{n-1}$ to $f_n$ belongs to the group
  \[
  \Ext^2_{\oh_{W_{0}}}\bigl(\LL_{W_0/\stX_0},f_0^*(\cI^n/\cI^{n+1})\bigr),
  \]
  where $\cI$ is the coherent ideal sheaf defining $\stX_0 \hookrightarrow \stX$. 
  This group is zero since $\LL_{W_0/\stX_0} = \Omega_{W_0/\stX_0}[0]$ is a vector bundle 
  and $W_0$ is affine.
  
Since $W_0$ is affine, each $W_n$ is also affine~\cite[Cor.~3.6]{knutson_alg_spaces}, \cite[Cor.~8.2]{rydh_noetherian-approx}.  It follows from \cite[Cor.~0.7.2.8]{egaI} that
$Z:= \Spec \big( \varprojlim_n \Gamma(W_n,\oh_{W_n}) \big)$ is a noetherian 
affine scheme complete along $W_0$ such that
$W_i$ is the $i$th infinitesimal neighborhood of $W_0$ in $Z$.
By Tannaka duality~\cite{hall-rydh_coherent-tannaka-duality}, there is an
induced morphism $\eta \co Z\to \stX$ which is formally versal at $W_0$
(\Cref{L:formally_versal}). Note that Tannaka duality applies because we
assume that $\stX$ has affine stabilizers.

\smallskip
\noindent {\bf Step 3: Existence of \'etale neighborhoods.} 
Applying  
  Artin algebraization for pairs (\Cref{T:artin-algebraization-intro}) yields an affine scheme $W$ of finite type over $\Spec \ZZ$, a closed immersion $W_0 \inj W$, an isomorphism $\widehat{W} \to Z$, and a morphism $f \co W \to \stX$ extending $\eta|_{W_n}$ for all $n$; in particular, $f \co W \to \stX$ is \'etale (resp.\ smooth) along $W_0$.  The preimage $f^{-1}(\stX_0)$ is a closed subscheme of $W$ which agrees with $W_0$ after restricting to the Zariski-localization of $W$ along $W_0$.  Therefore, there is an affine open subscheme $W' \subset W$ containing $W_0$ such that $f|_{W'}$ extends $f_0$. This finishes the
proof of \Cref{T:affine-etale-nbhds}.
\end{proof}

\end{section}


\begin{section}{Existence of geometric pushouts}
In this section, we prove \Cref{T:Ferrand}, on the existence of pushouts of algebraic stacks. The exposition will follow \cite[App.\ A]{hall_coherent-versality} closely, where a useful special case of this result was established. We begin with a definition.
\begin{definition}
Fix a $2$-commutative square of algebraic stacks 
  \[
  \xymatrix{\stX_0 \ar@{^(->}[r]^{{i}} \ar[d]_f
    \drtwocell<\omit>{\alpha} & \stX_1 \ar[d]^{f'}\\ \stX_2
    \ar@{^(->}[r]_{{i}'} & \stX_3,}
  \]
  where ${i}$ and ${i}'$ are closed immersions and $f$ and $f'$
  are affine. If the induced map
  \[
  \sO_{\stX_3} \to {i}'_*\sO_{\stX_2} \times_{({i}'f)_*\sO_{\stX_0}}
  f'_*\sO_{\stX_1}
  \]
  is an isomorphism of sheaves, then we say that the square is a 
  \emph{geometric pushout}, and that $\stX_3$
  is a \emph{geometric pushout} of the diagram $[\stX_2 \xleftarrow{f}
  \stX_0 \xrightarrow{{i}} \stX_1]$. 
\end{definition}
The main result of this section is the following refinement of \Cref{T:Ferrand}. It also generalizes \cite[Prop.\ A.2]{hall_coherent-versality} from the case of a locally nilpotent closed immersion to a general closed immersion. 
\begin{theorem}\label{T:geometric-pushout}
Any diagram of algebraic stacks $[\stX_2 \xleftarrow{f} \stX_0 \xrightarrow{i} \stX_1]$, where $i$ is a closed immersion, $f$ is affine, and $\stX_1$ is quasi-separated, admits a geometric pushout $\stX_3$. The resulting geometric pushout square is $2$-cartesian and $2$-cocartesian in the $2$-category of algebraic stacks with quasi-separated diagonals. If $\stX_1$ and $\stX_2$ are quasi-compact (resp.\ quasi-separated, Deligne--Mumford, algebraic spaces, affine schemes), then so is $\stX_3$. 
\end{theorem}
We will need the following two lemmas---the first is precisely 
\cite[Lem.\ A.3]{hall_coherent-versality} and the second is a mild extension of  
\cite[Lem.\ A.4]{hall_coherent-versality}.
\begin{lemma}\label{lem:geom_po_ferr_new}
  Fix a $2$-commutative square of algebraic stacks
    \[
  \xymatrix@-0.8pc{\stX_0 \ar@{^(->}[r]^{{i}} \ar[d]_f
    \drtwocell<\omit>{} & \stX_1 \ar[d]^{f'}\\ \stX_2
    \ar@{^(->}[r]_{{i}'} & \stX_3.}
  \]
  \begin{enumerate}
  \item\label{lem:geom_po_ferr_new:item:2cart} If the square is a
    geometric pushout, then it is $2$-cartesian. 
  \item\label{lem:geom_po_ferr_new:item:flb_gpo} If the square
    is a geometric pushout, then it 
    remains so after flat base change on $\stX_3$.
  \item\label{lem:geom_po_ferr_new:item:fl_loc_gpo} If after
    fppf base change on $\stX_3$ the square  
    is a geometric pushout, then it was a geometric pushout prior to
    base change.
  \end{enumerate}
 \end{lemma}
\begin{proof}
  The claim
  \itemref{lem:geom_po_ferr_new:item:2cart} is local on $\stX_3$ for the
  smooth topology, thus we may assume that everything in sight is
  affine---whence the result follows from \cite[Thm.\ 2.2]{ferrand_conducteur}.
  Claims \itemref{lem:geom_po_ferr_new:item:flb_gpo} and
  \itemref{lem:geom_po_ferr_new:item:fl_loc_gpo} are trivial
  applications of flat descent. 
\end{proof}
\begin{lemma}\label{lem:sadm_fl_loc_SA3}
  Consider a $2$-commutative diagram of algebraic stacks 
  \[
  \xymatrix@R-1.6pc@C-.5pc{   & \stU_0 \ar[dd]|!{[dl];[d]}\hole \ar[dl]
    \ar@{^(->}[rr] & & 
    \stU_1 \ar[dl] 
    \ar[dd] \\ \stU_2
    \ar[dd]  \ar@{^(->}[rr] & & \stU_3 \ar[dd]  &   \\   & \stX_0
    \ar@{^(->}[rr]|!{[r];[dr]}\hole 
    \ar[dl] & & \stX_1 \ar[dl]\\ 
    \stX_2 \ar@{^(->}[rr] & &\stX_3  &   } 
  \]
  where the back and left faces of the cube are $2$-cartesian and the
  top and bottom faces are geometric pushout squares. Then all
  faces of the cube are $2$-cartesian. Moreover, if the morphisms
  $\stU_1 \to \stX_1$ and $\stU_2 \to \stX_2$ have one of the following properties:
  \begin{enumerate}
  \item \label{lem:sadm_fl_loc_SA3:flat} flat;
  \item \label{lem:sadm_fl_loc_SA3:surj} surjective;
  \item \label{lem:sadm_fl_loc_SA3:ft} locally of finite type;
  \item \label{lem:sadm_fl_loc_SA3:fp} flat and locally of finite presentation; or
  \item \label{lem:sadm_fl_loc_SA3:smooth} smooth.
  \end{enumerate} 
   then the morphism $\stU_3 \to \stX_3$
  has the same property. 
\end{lemma}
\begin{proof}
  By \Cref{lem:geom_po_ferr_new}\itemref{lem:geom_po_ferr_new:item:flb_gpo},
  this is all smooth local on $\stX_3$ and $\stU_3$; thus, we immediately 
  reduce to the case where everything in sight is affine. Fix a
  diagram of rings $[A_2 \rightarrow A_0 \xleftarrow{p} A_1]$ where $p
  \colon A_1\to A_0$ is surjective. For $j=0$, $1$, $2$ fix 
  $A_j$-algebras $B_j$ and $A_0$-isomorphisms
  $B_2\otimes_{A_2} A_0 \cong B_0$ and $B_1\otimes_{A_1} A_0 \cong
  B_0$. Set $A_3 = A_2\times_{A_0} A_1$ and $B_3 = B_2\times_{B_0} 
  B_1$; then we first have to prove that
  the natural maps $B_3\otimes_{A_3} A_j \to B_j$ are
  isomorphisms, and that these isomorphisms are compatible with the
  given isomorphisms. This is an immediate consequence of \cite[Thm.\
  2.2(i)]{ferrand_conducteur}, since these are just questions about modules.  
  
  Case \itemref{lem:sadm_fl_loc_SA3:flat} similarly follows from \cite[Thm.\
  2.2(iv)]{ferrand_conducteur}. Case \itemref{lem:sadm_fl_loc_SA3:surj} follows from the 
  observation that $|\stX_1| \amalg |\stX_2| \to |\stX_3|$ and 
  $|\stU_1| \amalg |\stU_2| \to |\stU_3|$ are surjective \cite[Sch.\ 4.3 \& Thm.\ 5.1]{ferrand_conducteur}. Case \itemref{lem:sadm_fl_loc_SA3:smooth} follows from \itemref{lem:sadm_fl_loc_SA3:fp}, the surjectivity of $|\stX_1| \amalg |\stX_2| \to |\stX_3|$ already remarked, and the observation that smoothness is a fibral criterion for morphisms that are flat and locally of finite presentation.  
  
  For \itemref{lem:sadm_fl_loc_SA3:ft}, we argue as follows: by \cite[Thm.\ 2.2(ii)]{ferrand_conducteur}, an $A_3$-module $W_3$ is zero if
and only if the modules $W_3\otimes_{A_3} A_1$ and $W_3\otimes_{A_3} A_2$ are
zero. Now write $B_3$ as the union of its finite type $A_3$-sub-algebras $B_{3,\lambda}$. As
filtered direct limits commute with tensor products, it follows that for
sufficiently large $\lambda$, the homomorphisms $B_{3,\lambda}\otimes_{A_3} A_1  \to B_1$ and $B_{3,\lambda}\otimes_{A_3} A_2  \to B_2$ are
surjective. Looking at the cokernel, it follows that $B_{3,\lambda} \to B_3$ is
surjective.

For  \itemref{lem:sadm_fl_loc_SA3:fp}: if $B_j$ is a flat $A_j$-algebra of finite presentation for $j=1$, $2$, then we know by \itemref{lem:sadm_fl_loc_SA3:ft} that $B_3$ is of finite type. Hence, we can choose a surjection
$P_3=A_3[x_1,\ldots,x_n] \twoheadrightarrow B_3$. Let $J_3$ be the kernel. Since $B_3$ is $A_3$-flat, the sequence
\[
0 \to J_3\to P_3\to B_3 \to 0
\]
remains exact after tensoring by any $A_3$-algebra. In particular, $J_j = J_3 \otimes_{A_3} A_j$ is a $P_j = P_3 \otimes_{A_3} A_j$-module of finite type for $j=1$, $2$. It now follows from Ferrand's case of finite type modules (over the cocartesian
square defined by the $P_j$) that $J_3$ is a $P_3$-module of finite type;  hence $B_3$ is an $A_3$-algebra of finite presentation.
\end{proof}

We now come to an important lemma, where we make use of \Cref{T:affine-etale-nbhds} in a critical way. Note that the proof is almost identical to \cite[Lem.\ A.8]{hall_coherent-versality}.
\begin{lemma}\label{lem:sch_pushouts1}
  Fix a $2$-commutative square of algebraic stacks
  \[
  \xymatrix{\stX_0 \ar@{^(->}[r]^{{i}} \ar[d]_f
    \drtwocell<\omit>{\alpha} & \stX_1 \ar[d]^{f'}\\ \stX_2
    \ar@{^(->}[r]_{{i}'} & \stX_3.}
  \]
  If the square is a geometric pushout and $i$ is a 
  closed immersion, then the square is $2$-cartesian and
  $2$-cocartesian in the $2$-category of algebraic stacks with quasi-separated
  diagonals.
\end{lemma}
\begin{proof}
  That the square is $2$-cartesian is 
  \Cref{lem:geom_po_ferr_new}\itemref{lem:geom_po_ferr_new:item:2cart}. It
  remains to show that we can uniquely complete all $2$-commutative
  diagrams of algebraic stacks
  \vspace{-8mm}
  \[
  \xymatrix@ur{\stX_0 \ar@{^(->}[r]^{{i}} \ar[d]_f
    \drtwocell<\omit>{\alpha} & \stX_1 
    \ar[d]^{f'} \ar@/^1pc/[ddr]^{\psi_1} & \\  \stX_2
    \ar@{^(->}[r]_{{i}'} \ar@/_1pc/[drr]_{\psi_2} & \stX_3
    \drtwocell<\omit>{\beta} & \\ & & \stW}
  \vspace{-8mm}
  \]
  with a map $\stX_3\to \stW$ and compatible $2$-isomorphisms.
  By smooth descent, this is smooth-local on $\stX_3$, so we may reduce
  to the situation where the $\stX_j=\Spec A_j$ are all affine
  schemes. Since $\stX_3$ is a geometric pushout of the diagram
  $[\stX_2\xleftarrow{f} \stX_0 \xrightarrow{i} \stX_1]$, it follows that
  $A_3\cong A_2\times_{A_0} A_1$. 

  Let $q \colon \Spec B \to \stW$ be a smooth morphism such that the
  pullback $v_j \colon U_j \to \stX_j$ of $q$ along $\psi_j$ is
  surjective for $j\in \{0,1,2\}$, which exists because the $\stX_j$ are
  all quasi-compact. There are compatibly induced morphisms of
  quasi-separated algebraic spaces
  $\psi_{j,B} \colon U_j \to \Spec B$ for $j=1$ and
  $2$ and $f_B \colon U_0 \to U_2$ and $i_B \colon U_0 \hookrightarrow
  U_1$.

  Let $c_2 \colon \Spec C_2 \to U_2$ be an \'etale morphism such that
  $v_2\circ c_2$ is smooth and surjective. The morphism $c_2$ pulls
  back along $f_B$ to give an \'etale morphism $c_0\colon \Spec C_0 \to
  U_0$ such that $v_0\circ c_0$ is smooth and surjective. Let
  $\tilde{f} \colon \Spec C_0 \to \Spec C_2$ and $\tilde{\psi}_2
  \colon \Spec C_2 \to \Spec B$ be the resulting morphisms. 

  Since $c_0$ is \'etale and $i$ is a closed
  immersion, it follows that there is an \'etale morphism $c_1\colon \Spec C_1 \to
  U_1$ whose pullback along $i_B$ is isomorphic to $c_0$ (\Cref{T:affine-etale-nbhds}). It can easily be arranged that $v_1 \circ c_1$ is smooth and surjective. Let $C_3 = C_2\times_{C_0} C_1$. Then there
  is a uniquely induced ring homomorphism $A_3 \to C_3$. By
  \Cref{lem:sadm_fl_loc_SA3}, the morphism $c_3\colon \Spec C_3 \to
  \Spec A_3$ is smooth and surjective. Hence, we may replace $\Spec A_j$ by $\Spec C_j$ and
  further assume that the $\psi_j$ for $j=0$, $1$, and $2$ factor
  through some smooth morphism $q\colon \Spec B \to \stW$. In particular,
  there is an induced morphism $\psi_3\colon \Spec A_3 \to \Spec B \to
  \stW$. 
  
  It remains to prove that the morphism $\psi_3$ is unique up to a
  unique choice of $2$-morphism. Let $\psi_3$ and $\psi_3'\colon \Spec
  A_3 \to \stW$ be two compatible morphisms. That these morphisms
  are isomorphic can be checked smooth-locally on $\Spec A_3$. But
  smooth-locally, the morphisms $\psi_3$ and $\psi'_3$ both factor
  through some $\Spec B \to \stW$ and the morphisms $\Spec A_j \to \Spec
  A_3 \to \Spec B$ coincide for $j=0$, $1$, and $2$, thus $\psi_3$ and
  $\psi'_3$ are isomorphic. To show that the isomorphism between
  $\psi_3$ and $\psi'_3$ is unique, we just repeat the argument, and
  the result follows. 
\end{proof}
We finally come to the proof of  
\Cref{T:geometric-pushout}.
\begin{proof}[Proof of \Cref{T:geometric-pushout}]
  By \Cref{lem:sch_pushouts1}, it suffices to prove the existence
  of geometric pushouts. Let $\Coll_0$ denote the category of affine
  schemes. For ${d}=1$, $2$, $3$, let $\Coll_{d}$ denote the full
  $2$-subcategory of the $2$-category of algebraic stacks with affine
  $d$th diagonal. Note that $\Coll_3$ is the full $2$-category of
  algebraic stacks. We will prove by induction on $d\geq 0$ that if
  $[\stX_2 \xleftarrow{f} \stX_0 \xrightarrow{i} \stX_1]$ belongs to $\Coll_d$
  and $\stX_1$ is quasi-separated,
  then it admits a geometric pushout. For the base case, where $d=0$ and
  $\stX_j = \Spec A_j$ is affine, take
  $\stX_3 = \Spec ( A_2\times_{A_0} A_1 )$ and the result is clear.

  Now let $d>0$ and assume that $[\stX_2 \xleftarrow{f} \stX_0
  \xrightarrow{i} \stX_1]$ belongs to $\Coll_d$. Fix a smooth surjection
  $\amalg_{l\in \Lambda} X_2^l \to \stX_2$, where $X_2^l$ is an affine
  scheme $\forall l\in \Lambda$. Set $X_0^l = X_2^l \times_{\stX_2} \stX_0$.
  As $f$ is affine, the scheme $X_0^l$ is also affine. 
  By \Cref{T:affine-etale-nbhds}, the resulting smooth surjection $X_0^l \to \stX_0$
  lifts to a smooth surjection $X^l_1 \to \stX_1$, with $X^l_1$ affine, and
  $X_0^l\cong X^l_1\times_{\stX_1} \stX_0$. For $j = 0$, $1$, and $2$ and
  $u$, $v$, $w \in \Lambda$ set $X_j^{uv}=X_j^u\times_{\stX_j} X_j^v$ and
  $X_j^{uvw} = X_j^u \times_{\stX_j} X_j^v \times_{\stX_j} X_j^w$. Note that
  for $j=0$, $1$, and $2$ and all 
  $u$, $v$, $w\in \Lambda$ we have $X_j^{uv}$, $X_j^{uvw} \in
  \Coll_{d-1}$. By the inductive hypothesis, for $I=u$, $uv$ or $uvw$,
  a geometric pushout $X_3^I$ of the diagram $[X_2^I \leftarrow X_0^I
  \rightarrow X_1^I]$ exists. By \Cref{lem:sch_pushouts1}, there
  are uniquely induced morphisms $X^{uv}_j \to X^{u}_j$. For $j\neq
  3$, these morphisms are clearly smooth, and by 
  \Cref{lem:sadm_fl_loc_SA3} the morphisms
  $X^{uv}_3 \to X^u_3$ are smooth. It easily verified that the
  universal properties give rise to a smooth groupoid $[\amalg_{u,v\in
    \Lambda} X^{uv}_3 \rightrightarrows \amalg_{w\in \Lambda}
  X_3^w]$. The quotient $\stX_3$ of this groupoid in the category of
  stacks is algebraic. By 
  \Cref{lem:geom_po_ferr_new}\itemref{lem:geom_po_ferr_new:item:fl_loc_gpo}
  it is also a geometric pushout of the diagram $[\stX_2 \leftarrow \stX_0
  \rightarrow \stX_1]$ and the result follows.

  That the pushout inherits the properties ``quasi-compact'' and
  ``quasi-separated'' follows from $\stX_1\amalg \stX_2\to \stX_3$ being affine
  and surjective.  The properties ``Deligne--Mumford'' and ``algebraic
  space'', are inherited since $(\stX_1\smallsetminus\stX_0) \amalg \stX_2\to
  \stX_3$ is a surjective monomorphism. For ``affine'', this was the
  base case of the induction.
\end{proof}
\end{section}


\begin{section}{Local structure of algebraic stacks}\label{S:local-structure}
In this section we prove the main local structure results for stacks (\Cref{T:local-structure:excellent,T:local-structure:non-closed}) as well as non-noetherian generalizations (\Cref{T:local-structure:non-noeth,T:local-structure:non-noeth:pro-immersion}).  

\subsection{Proof of \texorpdfstring{\Cref{T:local-structure:excellent}}{Theorem~\ref{T:local-structure:excellent}}}

When $f_0 \co \stW_0 \to \stX_0$ is smooth or \'etale, the theorem can be established along similar lines to \cite[Proof of Thm.~12.1]{alper-hall-rydh_etale-local-stacks}.

\begin{proof}[Proof of \Cref{T:local-structure:excellent}\itemref{TI:local-structure:excellent1}---smooth/\'etale case] \qquad

\smallskip
\noindent
\textbf{Step 1: An effective formally versal solution.}
Since $\stW_0$ is quasi-compact, we may assume that $\stX$ is quasi-compact 
after replacing $\stX$ with a quasi-compact open substack containing the image of $\stW_0$.
Since $\stW_0$ is linearly fundamental, we can apply \cite[Thm.~1.11]{alper-hall-rydh_etale-local-stacks} to obtain a cartesian square
\[\xymatrix{
  \stW_0 \ar@{^(->}[r] \ar[d]^{f_0}  & \widehat{\stW} \ar[d]^{\widehat{f}} \\
  \stX_0 \ar@{^(->}[r]  & \stX   
}\]
where $\widehat{f} \co \widehat{\stW} \to \stX$ is a flat morphism and $\widehat{\stW}$ is a linearly fundamental stack coherently complete along $\stW_0$.
Since $\stW_n\to \stX_n$ is smooth, $\widehat{f} \co \widehat{\stW}\to
\stX$ is formally versal at $\stW_0$ (\Cref{L:formally_versal}).

\smallskip
\noindent
\textbf{Step 2: Algebraization.}  We now apply algebraization for linearly 
fundamental pairs (\Cref{T:Artin-algebraization}) to the pair $(\widehat{\stW}, \stW_0)$ and morphism $\widehat{f} \co \widehat{\stW}\to
\stX$ to obtain a fundamental
pair $(\stW, \stW_0)$ with $\stW$ of finite type over $S$ and a morphism $f \co \stW \to \stX$ smooth (resp.\ \'etale) along $\stW_0$ such that $\widehat{\stW}$ is isomorphic over $\stX$ to the coherent completion of $\stW$ along $\stW_0$.   After replacing $\stW$ with an open neighborhood, we may arrange that $f$ is smooth (resp.\ \'etale) and $\stW$ is fundamental.  Indeed, if $\stU \subset \stW$ is an open neighborhood of $\stW_0$ such that $f|_{\stU}$ is smooth (resp.\ \'etale) and if $\pi\colon \stW\to W$ denotes the adequate moduli space, then we replace $\stW$ with the inverse image of any affine open subscheme of $W\smallsetminus \pi(\stW\smallsetminus \stU)\bigr)$ containing $\pi(\stW_0)$. 
\end{proof}

The case when $f_0 \co \stW_0 \to \stX_0$ is syntomic is handled by reducing to the smooth case.

\begin{proof}[Proof of \Cref{T:local-structure:excellent}\itemref{TI:local-structure:excellent2}---syntomic case] \qquad

  \smallskip
  \noindent
  \textbf{Step 1: We may assume that $f_0 \co \stW_0 \to \stX_0$ is affine.}  
  We may assume that $\stX$ is quasi-compact.  Since $\stW_0$ is fundamental, there is an affine morphism $\stW_0 \to B \GL_n$ for some $n$.  Since $\stX_0$ has affine diagonal (as it has the resolution property), the induced morphism $\stW_0 \to \stX_0 \times B \GL_n$ is affine.  Since $B \GL_n$ is smooth with smooth diagonal, we may replace $(\stX, \stX_0)$ with $(\stX \times B \GL_n, \stX_0 \times B \GL_n)$. 

  \smallskip
  \noindent
  \textbf{Step 2: There is a factorization $f_0 \co \stW_0 \inj \stY_0 \to \stX_0$ where $\stW_0 \inj \stY_0$ is a regular closed immersion, $\stY_0 \to \stX_0$ is smooth and affine, and $\stY_0$ is linearly fundamental.}  
  Since $\stX_0$ has the resolution property and $(f_0)_* \oh_{\stW_0}$ is a finite type $\oh_{\stX_0}$-algebra, there exists a vector bundle $\cE_0$ on $\stX_0$ and a surjection $\Sym(\sE_0)\surj (f_0)_*\oh_{\stW_0}$.  Setting $\stY_0 = \VV(\sE_0)=\Spec(\Sym\sE_0)$ yields a factorization such that $\stW_0 \inj \stY_0$ is regular closed immersion and $\stY_0 \to \stX_0$ is smooth and affine.  To arrange that $\stY_0$ is linearly fundamental, we apply the \'etale case of the local structure theorem (\Cref{T:local-structure:excellent}\itemref{TI:local-structure:excellent1}) to the closed immersion $\stW_0\inj \stY_0$ to extend the isomorphism $\stW_0 \iso \stW_0$ to an \'etale morphism $\stY'_0 \to \stY_0$ with $\stY'_0$ fundamental. Since $\stW_0$ satisfies \ref{Cond:PC} or \ref{Cond:N}, or $\stY_0$ satisfies
  \ref{Cond:FC}, there is an open neighborhood of $\stW_0\inj \stY'_0$ that is
  linearly fundamental \cite[Prop.~16.14]{alper-hall-rydh_etale-local-stacks}.

  \smallskip
  \noindent
  \textbf{Step 3: Apply the smooth version of the local structure theorem.}  Since $\stY_0$ is linearly fundamental, we may apply the smooth case of the local structure theorem (\Cref{T:local-structure:excellent}\itemref{TI:local-structure:excellent1})
  to the closed immersion $\stX_0 \inj \stX$ and smooth morphism $\stY_0 \to \stX_0$ to obtain a commutative diagram
  \[
    \xymatrix{
      \stW_0 \ar@{^(->}[r] \ar[rd]_{f_0}  & \stY_0 \ar[d] \ar@{^(->}[r]   & \stY \ar[d] \\
                                          & \stX_0 \ar@{^(->}[r]          & \stX \ar@{}[ul]|\square
    }
  \] 
  with a cartesian square such that $\stY \to \stX$ is smooth and $\stY$ is fundamental.

  \textbf{Step 4: Lift the closed substack $\stW_0 \inj \stY_0$ to a closed substack $\stW \inj \stY$ syntomic over $\stX$.} 
  Let $\sI_0$ be the ideal sheaf defining $\stW_0 \inj \stY_0$ and consider
  the conormal bundle $\sN_0:=\sI_0/\sI_0^2$. After replacing $\stY$ with an \'etale
  fundamental neighborhood of $\stW_0\inj \stY$, we may extend the conormal
  bundle to a vector bundle $\sN$ on
  $\stY$; this follows from applying 
  \cite[Prop.~16.12]{alper-hall-rydh_etale-local-stacks} to the
  fundamental pair $(\stY, \stW_0)$ and
  the morphism $\stW_0 \to B \GL_n$ induced from $\sN_0$.

  We claim that after replacing $\stY$ with a fundamental \'etale neighborhood of $\stW_0$ the canonical homomorphism $\sN\surj 
  \sN_0 \inj \sO_{\stY_0}/\sI_0^2$ extends to a diagram
  \[
    \xymatrix{
      \cN \ar@{-->}[r] \ar@{->>}[d]        & \oh_{\stY} \ar@{->>}[d]   \\
      \cN_0 \ar@{^(->}[r]                        & \oh_{\stY_0}/\sI_0^2.
    }
  \]
  If $\stY$ is linearly fundamental, this is immediate as the functor $\Hom_{\oh_{\stY}}(\cN,-)
  = \Gamma(\stY, \cN^{\vee} \tensor -)$ is exact. 
  In general, let $\stW_1 \inj \stY_0$ be the closed substack defined by $\sI_0^2$. Then we have a morphism $\stW_1\to \VV(\cN)$ over $\stY$ which by \cite[Prop.~16.12]{alper-hall-rydh_etale-local-stacks} extends to a section
$\stY\to \VV(\cN)$ after replacing $\stY$ with a fundamental \'etale neighborhood of $\stW_0$. This gives the requested map $\cN\to \oh_{\stY}$.
  
  Let $\stW$ be the closed substack defined by the image of $\sN\to
  \oh_{\stY}$. By construction $\stW$ contains the closed substack $\stW_0$.
  We claim that $\stW_0\inj \stW\times_\stX \stX_0$ is an isomorphism and
  $\stW \to \stX$ is syntomic in an
  open neighborhood of $\stW_0$.  This establishes the theorem as we may shrink
  further to arrange that $\stW \to \stX$ is syntomic in a fundamental
  open neighborhood of $\stW_0$.
  These claims can be verified smooth-locally on $\stX$ and
  $\stY$ so we may assume that $\stX$ and $\stY$ are affine schemes and $\sN$
  is a trivial vector bundle. By construction $\sN\oh_{\stY_0}+\sJ_0^2=\sJ_0$
  so it follows that $\stW_0\inj \stW\times_\stY \stY_0$ is an isomorphism
  in an open neighborhood of $\stW_0$ by Nakayama's lemma.

  Let $f_1,\dots,f_n\in \sO_\stY$ be the image
  of a basis of $\sN$. We claim that $f_1,\dots,f_n$ is a regular sequence
  in a neighborhood of $\stW_0$ and that $\stW\to \stX$ is flat in a
  neighborhood of $\stW_0$. By~\cite[Thm.~11.3.8 (c)$\implies$(b')]{egaIV}, it
  is enough to prove that the images of $f_1,\dots,f_n$ in $\sO_{\stY_x,w}$
  is a regular sequence for every $w\in |\stW_0|$ with image $x\in |\stX_0|$, which
   follows by construction.
\end{proof}

\subsection{Non-noetherian local structure theorem}

The following provides a non-noetherian generalization of 
\Cref{T:local-structure:excellent}, which we
establish by reducing to the noetherian case.

\begin{theorem}[Local structure of stacks]\label{T:local-structure:non-noeth}
  Let $\stX$ be a quasi-separated algebraic stack with affine stabilizers, 
  $\stX_0 \inj \stX$ be a closed substack and 
  $f_0\colon \stW_0\to \stX_0$ be a morphism with $\stW_0$ linearly fundamental.
  Assume one of the following conditions:
  \begin{enumerate}
    \item \label{TI:local-structure:non-noeth1}
      $\stX$ is locally of finite type over an excellent algebraic space; or 
    \item \label{TI:local-structure:non-noeth2} 
      $\stW_0$ satisfies \ref{Cond:PC} or \ref{Cond:N}; or
    \item \label{TI:local-structure:non-noeth3}
      $\stX_0$ satisfies \ref{Cond:FC}.
  \end{enumerate}
  Then
  \begin{enumerate}[label=(\alph*)]
    \item \label{TI:local-structure:non-noeth:smooth-et} If $f_0$ is smooth (resp.\ \'etale), then there exists a smooth
      (resp.\ \'etale) morphism $f\colon \stW\to \stX$ such that $\stW$ is
      fundamental and $f|_{\stX_0}\simeq f_0$.
    \item \label{TI:local-structure:non-noeth:syntomic}  Assume that $\stW_0$ satisfies \ref{Cond:PC} or \ref{Cond:N} or $\stX_0$ satisfies
      \ref{Cond:FC}. If $f_0$ is syntomic and $\stX_0$ has the resolution property, then
      there exists a syntomic morphism $f\colon \stW\to \stX$ such that $\stW$ is
      fundamental and $f|_{\stX_0}\cong f_0$.
  \end{enumerate}
\end{theorem}

\begin{proof}
  Case \itemref{TI:local-structure:non-noeth1} is precisely \Cref{T:local-structure:excellent}.  For cases \itemref{TI:local-structure:non-noeth2}--\itemref{TI:local-structure:non-noeth3}, after replacing $\stX$ with a quasi-compact open substack we can assume
  that $\stX$ is quasi-compact.  If $\stX_0$ satisfies \ref{Cond:FC}, we can assume that $\stX$ is also \ref{Cond:FC}.  Indeed, let $S$ be the spectrum of $\ZZ$
  localized in the characteristics of $\stX_0$ and replace $\stX$, $\stX_0$ and $\stW_0$ with their base changes along $S\to \Spec \ZZ$.  Once the theorem is established in this case, we can use standard limit methods to replace $S$ with an open subscheme of $\Spec \ZZ$. If instead $\stX_0$ satisfies \ref{Cond:PC} or
  \ref{Cond:N}, we let $S=\Spec \ZZ$.

  By \cite{rydh_noetherian-general}, we can write $\stX_0\inj \stX$ as a limit
  of finitely presented closed immersions $\stX_{0,\lambda}\inj \stX$ with
  transition maps that are closed immersions. For sufficiently large $\lambda$,
  we can extend $f_0\colon \stW_0\to \stX_0$ to a map $f_{0,\lambda}\colon
  \stW_{0,\lambda}\to \stX_{0,\lambda}$ of finite presentation.
  For sufficiently large $\lambda$, we have that $f_{0,\lambda}$ is
  smooth/\'etale/syntomic. If $\stX_0$ has the resolution property,
  then so does $\stX_{0,\lambda}$ for sufficiently large $\lambda$. This follows
  from the Totaro--Gross characterization of the resolution property as having a
  quasi-affine morphism to $B\GL_N$ for some $N$~\cite{totaro_resolution-property,gross_tensor-generators} and
  \cite[Thm.~C]{rydh_noetherian-approx}.
  For sufficiently large $\lambda$, we also have that $\stW_{0,\lambda}$ is
  linearly fundamental~\cite[Thm.~15.3]{alper-hall-rydh_etale-local-stacks}
  using that either $\stW_0$ is \ref{Cond:PC} or \ref{Cond:N}, or $\stX$ is
  \ref{Cond:FC}.  After replacing $\stW_0$ and $\stX_0$ with $\stW_{0,\lambda}$
  and $\stX_{0,\lambda}$ we may thus assume that $\stX_0\inj \stX$ is of finite
  presentation.

  Using \cite{rydh_noetherian-general}, we may further write $\stX\to S$ as a
  limit of algebraic stacks $\stX_{\lambda}\to S$ of finite presentation.
  For sufficiently large $\lambda$, we can descend the finitely presented
  maps $f_0\colon \stW_0\to \stX_0$ and $i\colon \stX_0 \inj \stX$ to finitely
  presented maps $f_{0,\lambda}\colon \stW_{0,\lambda}\to \stX_{0,\lambda}$
  and $i_\lambda\colon \stX_{0,\lambda}\inj \stX_{\lambda}$. For sufficiently
  large $\lambda$, we have that $\stX_\lambda$ has affine
  stabilizers~\cite[Thm.~2.8]{hall-rydh_alg-groups-classifying} and,
  as before, that $f_{0,\lambda}$ is
  smooth/\'etale/syntomic, that $i_\lambda$ is a closed immersion, that
  $\stW_{0,\lambda}$ is linearly fundamental and that $\stX_{0,\lambda}$ has
  the resolution property.
  We are now in the situation of \Cref{T:local-structure:excellent}.
\end{proof}

\subsection{Compact generation}

We can now prove \Cref{T:compact-generation-derived-category}
on the compact generation of algebraic stacks in positive
characteristic.
\begin{proof}[Proof of \Cref{T:compact-generation-derived-category}]
  The implications
  \itemref{TI:compact-generation-derived-category:crisp}$\implies$\itemref{TI:compact-generation-derived-category:support}$\implies$\itemref{TI:compact-generation-derived-category:generated}
  and
  \itemref{TI:compact-generation-derived-category:points}$\implies$\itemref{TI:compact-generation-derived-category:closed-points}
  are trivial. The implication
  \itemref{TI:compact-generation-derived-category:generated}$\implies$\itemref{TI:compact-generation-derived-category:closed-points}
  is \cite[Thm.\ 1.1]{hall-neeman-rydh_no-compacts} since every closed point
  has positive characteristic. It remains to
  prove that
  \itemref{TI:compact-generation-derived-category:closed-points} implies
  \itemref{TI:compact-generation-derived-category:crisp}
  and \itemref{TI:compact-generation-derived-category:points}. To this end, let
  $x$ be a closed point of $\stX$, which we view as morphism
  $x\colon \Spec l \to \stX$, where $l$ is an algebraically closed
  field. Let $i_x \colon \mathcal{G}_x \inj \stX$ be the closed immersion
  of the residual gerbe of $x$. Then there
  is a field $\kappa(x)$ such that $\mathcal{G}_x \to \Spec \kappa(x)$
  is a coarse moduli space. Certainly, $\kappa(x) \subseteq l$. After
  taking a finite extension $\kappa(x) \subseteq k \subseteq l$,
  $(\mathcal{G}_x)_k \simeq BH$, for some group scheme $H$ over
  $k$. After passing to an additional finite extension of $k$, there
  is a subgroup scheme $H' \inj H$ such that
  $H'_l \simeq G^0_{\mathrm{red}}$. By assumption,
  $G^0_{\mathrm{red}}$ is a torus, so $H'$ is of multiplicative
  type. Set $\stW_0^x = BH'$, $\stX_0^x = \mathcal{G}_x$ and let
  $f_0^x \colon \stW_0^x \to \stX_0^x$ be the induced morphism. We
  claim that $f_0^x$ is syntomic. Indeed, $f_0^x$ is the composition
  $BH' \to BH \to \mathcal{G}_x$. Now $BH \to \mathcal{G}_x$ is the
  base change of $\Spec k \to \Spec \kappa(x)$, which is
  syntomic. Also, $BH' \to BH$ is fppf-locally the morphism
  $H/H' \to \Spec k$. Since $H \to \Spec k$ is syntomic
  (\Cref{L:syntomic-groups}) and $H \to H/H'$ is fppf,
  $H/H' \to \Spec k$ is syntomic. By descent, $BH' \to BH$ is syntomic
  and so $f_0^x$ is too.

  We now apply
  \Cref{T:local-structure:non-noeth}\itemref{TI:local-structure:non-noeth2}\itemref{TI:local-structure:non-noeth:syntomic}
  to $f_0^x$: this results in a syntomic morphism
  $f^x\colon \stW^x \to \stX$ such that $\stW^x$ is fundamental and
  $f^x|_{\stX_0^x} \simeq f_0^x$. Since $\stW^x$ is fundamental and
  $f_0^x$ is finite, we may shrink $\stW^x$ so that $f^x$ is
  quasi-finite \cite[Lem.\
  3.1]{alper-hall-rydh_luna-stacks}. Additionally, since $\stX$ has
  affine diagonal, we may further shrink $\stW^x$ so that $f^x$ is
  affine \cite[Prop.~12.5]{alper-hall-rydh_etale-local-stacks}. By
  \cite[Prop.\ 13.4]{alper-hall-rydh_etale-local-stacks}, after
  passing to a strictly \'etale neighborhood of $\stW_0^x$, we may further
  shrink $\stW^x$ so that it is nicely fundamental.

  Since $\stX$ is quasi-compact and the $f^x$ are all open morphisms,
  there is a finite set of closed points $x_1$, $\dots$, $x_m$ of $\stX$
  such that the induced morphism
  $f\colon \stW = \amalg_{i=1}^m \stW^{x_i} \xrightarrow{\amalg
    f^{x_i}} \stX$ is affine, quasi-finite, syntomic, and faithfully
  flat. But $\stW$ is nicely fundamental, so it is $\aleph_0$-crisp
  \cite[Ex.\ 8.6]{hall-rydh_perfect-complexes}. By \cite[Thm.\
  C]{hall-rydh_perfect-complexes}, $\stW$ is $\aleph_0$-crisp. This
  proves
  \itemref{TI:compact-generation-derived-category:closed-points}$\implies$\itemref{TI:compact-generation-derived-category:crisp}.
  Since $\stW\to \stX$ is quasi-finite and surjective and the reduced
  identity components of the stabilizers of $\stW$ are tori, so are those
  of $\stX$. This proves
  \itemref{TI:compact-generation-derived-category:closed-points}$\implies$\itemref{TI:compact-generation-derived-category:points}.  
\end{proof}
We include the following result below for lack of reference.
\begin{lemma}\label{L:syntomic-groups}
  Let $S$ be an algebraic space. If $G \to S$ is a group algebraic
  space that is flat and locally of finite presentation, then it is
  syntomic.
\end{lemma}
\begin{proof}
  Since $G \to S$ is flat and locally of finite presentation, we
  reduce immediately to the situation where $S$ is the spectrum of an
  algebraically closed field $k$ \cite[Tags \spref{01UF} \&
  \spref{069N}]{stacks-project}. Let $G^0 \subseteq G$ be the
  connected component of the identity, which is a normal, irreducible,
  and quasi-compact flat closed subgroup scheme of $G$ \cite[Tag
  \spref{0B7R}]{stacks-project}. Since $G$ is locally of finite type,
  $G^0 \subseteq G$ is even open and closed. Hence, the quotient
  $G/G^0$ is \'etale. Thus, we may replace $G$ with $G^0$ and assume
  that $G$ is connected and of finite type. If the characteristic of $k$
  is $0$, then $G \to S$ is smooth and we are done (Cartier's Theorem
  \cite[Tag \spref{047N}]{stacks-project}). In general, there is an
  extension of groups
  $1 \to G_{\mathrm{ant}} \to G \to G_{\mathrm{aff}} \to 1$, where
  $G_{\mathrm{ant}}$ is anti-affine (i.e.,
  $\Gamma(G_{\mathrm{ant}},\sO_{G_{\mathrm{ant}}}) \simeq k$) and
  $G_{\mathrm{aff}}$ is affine \cite[(0.2)]{brion_anti-affine}. Then
  $G_{\mathrm{ant}}$ is smooth, so it suffices to prove the claim when
  $G$ is affine. In this case, $G \inj \GL_n$ for some
  $n>0$. Then $\GL_n/G$ is smooth and so the morphism
  $\GL_n \to \GL_n/G$ is syntomic; hence, $G \to \Spec k$ is syntomic.
\end{proof}
\subsection{Local structure of stacks at pro-affine-immersions} 
\label{subsec:local-structure-pro-immersions}
We recall \cite[\S3]{temkin-tyomkin_Prufer}: a morphism of algebraic
stacks $j \colon \stU \to \stX$ is a \emph{pro-open immersion} if
every morphism $\stY \to \stX$ with set-theoretic image contained in
$|j(\stU)|$ factors uniquely through $j$. It is established in
\cite[Prop.\ 3.1.4]{temkin-tyomkin_Prufer} that $j$ is necessarily a
flat monomorphism and
$|j(\stU)| = \cap_{\stV \supseteq |j(\stU)|} \stV$, where the
intersection ranges over all open stacks $\stV \subseteq \stX$
containing $j(\stU)$. If $j$ is quasi-compact, then it is a pro-open
immersion if and only if it is a flat monomorphism \cite[Thm.\
3.2.5]{temkin-tyomkin_Prufer} and then $j$ is quasi-affine
\cite[Prop.\ 1.5 (ii)]{raynaud_sem-samuel}.
If $j$ is quasi-compact, then it is also a
topological embedding \cite[Prop.\ 1.2]{raynaud_sem-samuel} and
if in addition $\stX$ is quasi-compact, then
$|j(\stU)| = \cap_{\stV \supseteq |j(\stU)|} \stV$, where the
intersection ranges over the quasi-compact opens of $\stX$ containing
$j(\stU)$.
\begin{remark}
  Let $j\colon \stU \to \stX$ be a quasi-compact pro-open immersion of
  algebraic stacks. There is a factorization of $j$ as
  $\stU \xrightarrow{j'} \stX' \xrightarrow{g} \stX$, where $j'$ is an
  affine pro-open immersion and $g$ is a quasi-compact open immersion
  \cite[Prop.\ 1.5 (i)]{raynaud_sem-samuel}.
\end{remark}
We introduce the following variant: a morphism of algebraic stacks
$j \colon \stU \to \stX$ is a \emph{pro-affine(-open) immersion} if
$\stU$ represents a cofiltered intersection
$\cap_{\alpha} \stV_\alpha$, where the $\stV_\alpha \subseteq \stX$
are (open) immersions and the transition maps
$\stV_\alpha \to \stV_\beta$, which are automatically (open)
immersions, are eventually affine.
\begin{example}
  An immersion of algebraic stacks is a pro-affine-immersion.
\end{example}
\begin{example}
  If $x \in |\stX|$ is a point of a quasi-separated algebraic stack,
  then the inclusion $\stG_x \to \stX$ of the residual gerbe is a
  pro-affine-immersion \cite[Lem.\ 2.1]{MR3754421}.
\end{example}
\begin{remark}
  A pro-affine-open immersion of algebraic stacks is
  pro-\'etale.
\end{remark}
\begin{remark}
  If $\stX$ is a normal and $\QQ$-factorial noetherian stack, then any
  quasi-compact pro-open immersion $j \colon \stU \to \stX$ is
  pro-affine-open. This follows from the result \cite[Cor.~2.7]{raynaud_sem-samuel}: after
  restricting to an open substack, the complement of $\stU$ is a, possibly
  infinite, union of Cartier divisors and the complements of finite unions
  of these divisors are affine open immersions.
\end{remark}
The following theorem simultaneously generalizes
\Cref{T:local-structure:non-closed} and
\Cref{T:local-structure:non-noeth}. Note that in
\Cref{T:local-structure:non-closed} no extra conditions are needed as
\ref{Cond:FC} always holds for the residual gerbe as it is a one-point space.

\begin{theorem}[Local structure of stacks at
  pro-affine-immersions]\label{T:local-structure:non-noeth:pro-immersion}
Assumptions and conclusions as in \Cref{T:local-structure:non-noeth} \itemref{TI:local-structure:non-noeth2} or \itemref{TI:local-structure:non-noeth3}
except that $\stX_0\inj \stX$ is a pro-affine-immersion.
\end{theorem}

\begin{proof}
As a first preliminary step, we can as before assume that $\stX$ is
quasi-compact and, if $\stX_0$ satisfies \ref{Cond:FC},
that $\stX$ satisfies \ref{Cond:FC} by base changing
along $S \to \Spec \ZZ$ where $S$ is the spectrum of $\ZZ$
localized in the characteristics of $\stX_0$. 

By assumption, $\stX_0 = \cap_\lambda \stX_\lambda$ is an intersection
of a cofiltered system of immersions $\stX_\lambda \inj \stX$
with eventually affine inclusions
$\stX_\mu \inj \stX_{\lambda}$. Pick $\alpha$ sufficiently large
such that $\stX_\lambda \inj \stX_{\alpha}$ is affine for all
$\lambda\geq \alpha$ and pick a quasi-compact open neighborhood
$\stU$ of $\stX_0$ in $\stX_\alpha$. Then $\stX_0 = \cap_{\lambda\geq \alpha}
(\stX_\lambda\cap \stU)$ so we may assume that all the $\stX_\lambda$ are
quasi-compact and that all the $\stX_\mu \inj \stX_\lambda$ are affine.

By standard limit methods, the
morphism $f_0 \colon \stW_0 \to \stX_0$ descends to a morphism
$f_\alpha \colon \stW_\alpha \to \stX_\alpha$, which is \'etale, smooth or
syntomic if $f_0$ is so. If $\lambda\geq \alpha$, set
$\stW_{\lambda} = \stW_\alpha \times_{\stX_\alpha} \stX_{\lambda}$. Then
 $\stW_0 = \cap_{\lambda\geq \alpha} \stW_\lambda$. Now either
$\stX$ satisfies \ref{Cond:FC} (by the initial reduction) or
$\stW_0$ satisfies \ref{Cond:PC} or \ref{Cond:N}. Hence, 
$\stW_\beta$ is linearly fundamental for some $\beta \gg \alpha$
\cite[Prop.\ 15.3]{alper-hall-rydh_etale-local-stacks}.
After replacing $\stX$ with an open neighborhood of $\stX_\beta$,
we may assume that $\stX_\beta\inj \stX$ is a closed immersion.
We may now apply \Cref{T:local-structure:non-noeth}
\itemref{TI:local-structure:non-noeth2} or
\itemref{TI:local-structure:non-noeth3} to $f_\beta$ and the result follows.
\end{proof}
We now prove the refinements.
\begin{proof}[Proof of \Cref{T:refinement1}]
  Arguing as in the proof of
  \Cref{T:local-structure:non-noeth:pro-immersion}, we may assume that
  $\stW_0$ satisfies \ref{Cond:PC} or \ref{Cond:N} or $\stW$ satisfies
  \ref{Cond:FC}. We may further assume that there is a factorization
  of $\stW_0 \inj \stW \to \stX$ through an immersion
  $\stW_\beta \inj \stW$ such that $\stW_{\beta}$ is fundamental
  and $\stW_\beta \inj \stW \to \stX$ is
  representable~\cite[Thm.~C]{rydh_noetherian-approx}. We now factor
  $\stW_\beta \inj \stW$ as
  $\stW_\beta \inj \stZ \subseteq \stW$, where
  $\stW_\beta \inj \stZ$ is a closed immersion and
  $\stZ \subseteq \stW$ is an open immersion. In this generality,
  however, $\stZ$ is not necessarily fundamental (it can be arranged
  to be if $\stW_0 \inj \stW$ is a closed immersion,
  however). But we can now apply \Cref{T:local-structure:non-noeth} to
  the closed immersion $\stW_\beta \inj \stZ$. We thus obtain an
  \'etale neighborhood $p \colon \stW' \to \stZ$ of $\stW_\beta$ such
  that $\stW'$ is fundamental and the induced morphism
  $\stW_\beta' = p^{-1}(\stW_\beta) \simeq \stW_\beta \to \stZ
  \subseteq \stW \to \stX$ is representable. The
  result now follows from \cite[Prop.\
  12.5]{alper-hall-rydh_etale-local-stacks} applied to the pair
  $(\stW',\stW_\beta')$ and the morphism $\stW' \to \stX$.
\end{proof}
\begin{proof}[Proof of \Cref{T:refinement2}]
  As in the proof of \Cref{T:refinement2}, we may assume that there is a
  factorization of $\stW_0 \inj \stW$ through an immersion
  $\stW_\beta \inj \stW$ such that $\stW_{\beta}$ is linearly or
  nicely fundamental, respectively. Likewise, if $\stW_0=[\Spec A_0/G_0]$,
  then we can arrange so that $\stW_\beta=[\Spec A_\beta/G_\beta]$
  with $G_\beta$ linearly reductive or nice.

  We have a closed then open factorization $\stW_\beta \inj \stZ \subseteq
  \stW$. Apply \Cref{T:local-structure:non-noeth} to $\stW_\beta\inj \stZ$ to
  replace $\stW$ with an \'etale neighborhood of $\stW_\beta$ that is
  fundamental.  We can now apply \cite[Props.~16.11, 16.12 and
  16.14]{alper-hall-rydh_etale-local-stacks} and the result follows.
\end{proof}

\subsection{Nisnevich neighborhoods}
\begin{proof}[Proof of \Cref{T:nisnevich-neighborhoods}]
  We apply \Cref{T:local-structure:non-closed} to \emph{every} point
  of $|\stX|$: for each $x\in |\stX|$ we obtain an \'etale morphism
  $f_x \colon \stW_x \to \stX$ such that $f_x|_{\stG_x}$ is an
  isomorphism and $\stW_x$ is fundamental. If $\stX$ has affine (resp.\ separated) diagonal, then \Cref{T:refinement1} says that we can arrange that $f_x$ is affine (resp.\ representable). By \Cref{T:refinement2}, we may further assume
  that $\stW_x$ is nicely fundamental. Set
  $\stW = \amalg_{x\in |\stX|} \stW_x$ and take
  $f=\amalg_{x} f_x \colon \stW \to \stX$; then $f$ is a
  quasi-separated Nisnevich covering. By \cite[Prop.\ 3.3]{MR3754421},
  we may shrink $\stW$ so that it is quasi-compact (a monomorphic
  splitting sequence must factor through finitely many of the
  $\stW_x$), remains nicely fundamental and $f$ is a Nisnevich
  covering. 
\end{proof}
\begin{proof}[Proof of \Cref{T:nisnevich-neighborhoods-LR}]
  We apply \Cref{T:local-structure:non-noeth} to every closed point of
  $|\stX|$: for each closed point $x$ of $\stX$ we obtain an \'etale
  morphism $g \colon \stW \to \stX$ such
  that $\stW=[U/\GL_n]$ is fundamental and $g|_{\stG_x}$ is an
  isomorphism. If $\stX$ has affine (resp.\ separated diagonal), then
  \Cref{T:refinement1} says that we can arrange that $g$ is
  affine (resp.\ representable).

  For an integer $d\geq 1$, let $\stW^d$ be
  the $d$th fiber product of $g$; then the symmetric group $S_d$ acts
  on $\stW^d$ by permuting the factors. Let $e$ be the maximum rank of
  a fiber of $g$. Then there is an induced Nisnevich covering
  $f \colon \coprod_{1 \leq d \leq e} [\stW^d/S_d] \to \stX$ since
  $g$ is representable.

  Let $V^d$ be the $d$th fiber product of $U\to \stW\to \stX$.  Then
  $\stW^d=[V^d/(\GL_n)^d]$. Let $P$ be one of the properties: separated,
  quasi-affine, affine. If the diagonal of $\stX$ has property $P$, then
  the algebraic space $V^d$ has property $P$. Since the Stiefel manifold
  $\GL_{dn}/(\GL_n)^d$ is affine, it follows that $\stW^d=[V'/\GL_{nd}]$ for an
  algebraic space $V'$ with property $P$.

  Let $p\colon \stW^d\to [\stW^d/S_d]$. Let $\sE$ be the vector bundle on
  $\stW^d$ with frame bundle $V'$. Then we claim that the frame bundle $V$ of
  $p_*\sE$ is an algebraic space with property $P$. Indeed, $V$ is an algebraic
  space since the stabilizers of $[\stW^d/S_d]$ act faithfully on $p_*\sE$,
  cf.\ \cite[Lem.~2.13]{EHKV}. Since $p^*V\to V$ is finite, \'etale and
  surjective, it is enough to prove that $p^*V$ has property $P$.
  But since $p$ is finite \'etale, we have that $p^*p_*\sE\to \sE$ is split
  surjective and it follows that $p^*V$ has property $P$ by considering Stiefel
  manifolds again.  We have thus shown that $[\stW^d/S_d]=[V/\GL_N]$ for an
  algebraic space $V$ with property $P$.

  When $\stX$ has affine diagonal, then $\stW^d\to \stX$ is affine but
  $[\stW^d/S_d]\to \stX$ is merely separated.  Let $\SEC^d(\stW/\stX)\subseteq
  \stW^d$ be the open and closed substack that is the complement of all
  diagonals.  Then $S_d$ acts freely on $\SEC^d(\stW/\stX)$ relative to $\stX$
  and $\ET^d(\stW/\stX):=[\SEC^d(\stW/\stX)/S_d]\to \stX$ is affine and an
  \'etale neighborhood of any point of $\stX$ at which $g$ has rank $d$.  Thus
  $f\colon \coprod_{1 \leq d \leq e} \ET^d(\stW/\stX) \to \stX$ is a
  fundamental Nisnevich covering with $f$ affine.
\end{proof}

\subsection{Existence of henselizations}

\begin{proof}[Proof of \Cref{T:henselization}]
    By \Cref{T:local-structure:non-noeth:pro-immersion}, there exists an \'etale neighborhood $\stW \to \stX$ of $\stW_0:=\stX_0$ such that $\stW$ is fundamental.  Let $\stW_0 \to W_0$ and $\stW \to W$  be the good and adequate moduli spaces.  We claim that the henselization $W^h$ of $W$ along $W_0$
    exists and is affine. If $\stW_0$ is a closed substack, then this follows from \cite[Ch.~XI, Thm.~2]{raynaud_hensel_rings} as $W$ is affine. If $\stW_0 = \stG_x$ is the residual gerbe of a point $x \in |\stX|$, then $W_0 = \Spec \kappa(x) \inj W$ is the inclusion of a point $w$ and $W^h = \Spec \oh_{W,w}^h$.  In this case, we also note that $\stW_0$ satisfies \ref{Cond:FC}. Let $\stW^h = \stW \times_W W^h$.  Since $W^h \to W$ is flat, $\stW^h \to W^h$ is an adequate moduli space.  By \cite[Thm.~3.6]{alper-hall-rydh_etale-local-stacks},  $(\stW^h, \stW_0)$ is a henselian pair and by \cite[Thm.~13.7]{alper-hall-rydh_etale-local-stacks}, $\stW^h$ is linearly fundamental since the closed points of $\stW^h$ have linearly reductive stabilizer.   
  To show that $\stW^h \to \stX$ is the henselization of $\stX$ along $\nu \co \stX_0 \inj \stX$, 
  it is enough to prove that any quasi-separated \'etale neighborhood $g \co \stW'\to \stW^h$ of $\stW_0$ has a
  section. This is precisely the conclusion of \cite[Prop.~16.4]{alper-hall-rydh_etale-local-stacks}.  
\end{proof}

\end{section}

\begin{section}{Local structure of derived algebraic stacks}\label{S:local-structure-derived}
In this section we give a derived version of the local structure theorem.

An algebraic derived $1$-stack is the derived analogue of an algebraic stack: it is a sheaf of $\infty$-groupoids on the opposite of the $\infty$-category of simplicial commutative rings (with its \'etale topology) that admits a surjective morphism, represented by smooth derived algebraic spaces, from a disjoint union of derived affine schemes.

Let $\stX$ be an algebraic derived $1$-stack. We say that $\stX$ is \emph{fundamental} if there exists an affine morphism $\stX\to B\GL_n$; that is, if $\stX=[\Spec A/\GL_n]$ for some derived affine scheme $A$. We say that $\stX$ is \emph{linearly fundamental} if it is fundamental and cohomologically affine, that is, $R\Gamma(\stX,-)$ is $t$-exact.

\begin{proposition}[Derived effectivity theorem]\label{P:derived-effectivity}
Let $\stX_0\inj \stX_1\inj \stX_2\inj \dots$ be a sequence of derived
thickenings, i.e., $\stX_m \cong \tau_{\leq m} \stX_n$ for every $m\leq n$. If
$\stX_0$ is linearly fundamental, then there is a linearly fundamental algebraic derived $1$-stack $\stX$ and a compatible sequence of equivalences $\tau_{\leq n} \stX \cong \stX_n$.
\end{proposition}
\begin{proof}
The existence and uniqueness of an algebraic $1$-stack with compatible isomorphisms $\tau_{\leq n} \stX \cong \stX_n$ is given by \cite[Prop.~5.4.6]{Lurie-Thesis}. $\stX$ is determined by the equivalence $\stX(A) \cong \stX_n(A)$ for $n$-truncated simplicial commutative rings, and $\stX(A) \cong \varprojlim_n \stX(\tau_{\leq n} A)$ in general.

Let $f_0\colon \stX_0\to B\GL_r$ be an affine morphism. The obstruction to
lifting a morphism $f_n\colon \stX_n\to B\GL_r$ to $f_{n+1}\colon \stX_{n+1}\to
B\GL_r$ lies in
\[
\Ext^1_{\stX_0}\bigl(f_0^*\LL_{B\GL_r},\pi_{n+1}(\sO_{\stX_{n+1}})[n+1]\bigr).
\]
This obstruction group
vanishes since $\stX_0$ is linearly fundamental. We can thus find a compatible
sequence of morphisms $f_n\colon \stX_n\to B\GL_r$. Since $f_0$ is affine,
so is $f_n$ for every $n$. The compatible family of morphisms $f_n : \stX_n \to B\GL_r$ defines a morphism $f : \stX \to B\GL_r$, and the resulting morphism is affine because it is affine on every truncation.

Finally, because pushforward along the inclusion $\stX_0 \hookrightarrow \stX$ is $t$-exact and identifies $\QCoh(\stX)^\heart \cong \QCoh(\stX_0)^\heart$, $\stX$ is cohomologically affine if and only if $R\Gamma(\stX_0,-)$ has cohomological dimension $0$ on $\QCoh(\stX_0)^\heart$. Since $\stX_0$ has affine diagonal, this is the same as being cohomologically affine.
\end{proof}

\begin{proof}[Proof of \Cref{T:local-structure:derived}]
If $\stX_0$ satisfies \ref{Cond:FC}, then let $S$ be the spectrum of $\ZZ$ localized in
the characteristics of $\stX_0$ and base change everything along $S\to \Spec
\ZZ$. At the very end, we can then replace $S$ by an open quasi-compact
subscheme of $\ZZ$.

First assume that $f_0$ is smooth. Then $\stW_0\times_{\stX_0} (\stX_0)_\cl$ is
classical and we may apply the classical version of the local structure theorem
(\Cref{T:local-structure:non-noeth}). This gives us a fundamental
classical stack $\stW_\cl$ and a smooth morphism $f_\cl \colon \stW_\cl\to
\stX_\cl$.  Since either \ref{Cond:PC}/\ref{Cond:N} holds for $\stW_0$ or \ref{Cond:FC} for $\stW_\cl$, we may
assume that $\stW_\cl$ is linearly fundamental (\Cref{T:refinement2}).
We may now deform $f_{\leq 0}:=f_\cl$ to smooth maps $f_{\leq n} \colon \stW_{\leq n}\to \tau_{\leq n}\stX$ for every $n$. Indeed, the obstruction lies in
\[\Ext^2_{\stW_\cl}(\LL_{f_\cl},f_\cl^*\pi_n(\sO_\stX)[n]),\] 
which vanishes as $\stW_\cl$ is cohomologically affine and $f_\cl$ is smooth.
By \Cref{P:derived-effectivity}, there is a linearly fundamental derived $1$-stack $\stW$ with compatible isomorphisms $\stW_{\leq n} \cong \tau_{\leq n}\stW$. Because both $\stW$ and $\stX$ are nilcomplete \cite[Prop.~5.3.7]{Lurie-Thesis}, the smooth morphisms $\stW_{\leq n} \to \tau_{\leq n} \stX$ extend uniquely to a smooth morphism $f\colon \stW\to \stX$. Since
\[\Ext^1_{(\stW_0)_\cl}(\LL_{(f_0)_\cl},(f_0)_\cl^*\pi_n(\sO_{\stX_0})[n])=0,\]
the isomorphism $\stW_0\times_{\stX_0} (\stX_0)_\cl\to \stW\times_{\stX}
(\stX_0)_\cl$ extends to an isomorphism $\stW_0\to \stW\times_{\stX} \stX_0$
over $\stX_0$.

When instead $f_0$ is quasi-smooth, we proceed as in the syntomic case of the
classical version of the local structure
theorem, see the proof of \Cref{T:local-structure:excellent}\itemref{TI:local-structure:excellent2}.

\textbf{Step 1:} First replace $\stX$ with $\stX\times B\GL_n$ so that
$\stW_0\to \stX_0$ becomes affine.

\textbf{Step 2:} Consider the morphism of classical stacks
$(\stW_0)_\cl\to (\stX_0)_\cl$ and pick a factorization
$(\stW_0)_\cl\to \stY_0\to (\stX_0)_\cl$ where the first map is a closed
immersion and the second map is affine and smooth. Here we use that
$(\stX_0)_\cl$ has the resolution property.
Then apply the classical \'etale
version of the structure theorem to $(\stW_0)_\cl=(\stW_0)_\cl\inj \stY_0$.
We can thus replace $\stY_0$ with an \'etale neighborhood of $(\stW_0)_\cl$
and assume that $\stY_0$ is linearly fundamental.

\textbf{Step 3:} Apply the smooth case of the derived local structure theorem to
$\stY_0\to (\stX_0)_\cl\inj \stX$ and we obtain a smooth map $\stY\to \stX$.
Since $\stY\to \stX$ is smooth and $\stW_0$ is linearly fundamental,
the obstructions to lifting the closed immersion
$(\stW_0)_\cl\inj \stY$ to closed immersions $\tau_{\leq n}(\stW_0)\inj \stY$
over $\stX$ for every $n$ vanish. We obtain a closed immersion
$\stW_0\inj \stY$
because both $\stW_0$ and $\stY$ are nilcomplete \cite[Prop.~5.3.7]{Lurie-Thesis}.
Since either \ref{Cond:PC}/\ref{Cond:N} holds for $\stW_0$ or \ref{Cond:FC}
for $\stY$, we may assume that $\stY$ is linearly fundamental
(\Cref{T:refinement2}).

\textbf{Step 4:} Now let $\stY_0=\stY\times_{\stX} \stX_0$ (previously it denoted its
classical truncation). The morphism $\stW_0\to \stY_0$ is a quasi-smooth closed
immersion. Let $\sN=\pi_1(\LL_{\stW_0/\stY_0})$ denote the corresponding conormal
bundle on $(\stW_0)_\cl$. After replacing $\stY$ with an \'etale neighborhood
of $\stW_0$ \cite[Prop.~16.12]{alper-hall-rydh_etale-local-stacks}, we may
assume that $\sN$ extends to a vector bundle $\sE$ on $\stY$.

Let $F$ denote the homotopy fiber of $\sO_{\stY_0}\to \sO_{\stW_0}$. Since the
Hurewicz map $F\otimes_{\sO_{\stY_0}} \sO_{\stW_0}\to \LL_{\stW_0/\stY_0}[-1]$ is
an isomorphism on $\pi_0$, we have an induced isomorphism
$\sE|_{(\stW_0)_\cl}\simeq \sN\simeq F|_{(\stW_0)_\cl}$.
Since $\stY_0$ is cohomologically affine, this lifts to a
map $\sE|_{\stY_0}\to F$. The
composition $s_0\colon \sE|_{\stY_0}\to F\to \sO_{\stY_0}$ corresponds to a
section $s_0^\vee$ of $\sE^\vee|_{\stY_0}$ and the derived zero-locus of this
section $\stZ_0:=\{s_0^\vee=0\}\inj \stY_0$ defines a quasi-smooth closed
immersion. Here the derived zero-locus is the pull-back fitting in the
cartesian square
\[
\xymatrix{%
\stZ_0\ar[r]\ar[d] & \stY_0\ar[d]^{s_0^\vee} \\
\stY_0\ar[r]^-0 & \VV(\sE|_{\stY_0}).\ar@{}[ul]|\square
}
\]
The map $s_0\colon \sE|_{\stY_0}\to F$ corresponds to a $2$-commutative
diagram
\[
\xymatrix{%
\stW_0\ar[r]\ar[d]\drtwocell<\omit>{} & \stY_0\ar[d]^{s_0^\vee} \\
\stY_0\ar[r]^-0 & \VV(\sE|_{\stY_0}).
}
\]
and hence to a map $\stW_0\to \stZ_0$. By construction, we have that
$\LL_{\stW_0/\stZ_0}=0$ so the closed immersion $\stW_0\to \stZ_0$ is
also an open immersion. After replacing $\stY$ with an open neighborhood
of $\stW_0$, we can thus assume that $\stW_0=\stZ_0$.

Finally, we may lift the section $s_0^\vee$ of $\sE^\vee|_{\stY_0}$ to a section
$s^\vee$ of $\sE^\vee$ since $\stY$ is cohomologically affine. The derived
zero-locus $\stW:=\{s^\vee=0\}\inj \stY$ is a quasi-smooth closed immersion
restricting to $\stW_0\inj \stY_0$ and the composition $f\colon \stW\inj
\stY\to \stX$ is a quasi-smooth morphism such that $f|_{\stX_0} \simeq f_0$.
\end{proof}

\end{section}


\appendix
\begin{section}{Non-existence of Zariskification}\label{A:Zariskification}
In this section, we show that the Zariskification, in contrast to the
henselization, does not exist in general. This counter-example was mentioned
in~\cite[3.1.2]{temkin-tyomkin_Prufer}.

Let $X$ be a scheme and $Z\inj X$ be a closed subscheme. The generization of
$Z$ is the subset of $X$ consisting of all points $x\in |X|$ such that
$\overline{\{x\}}\cap Z\neq \emptyset$. A \emph{Zariskification} of $X$ along
$Z$ is a flat quasi-compact monomorphism $W\to X$ such that the image is the
generization of $Z$. The Zariskification is unique up to isomorphism since if
$W$ and $W'$ are two monomorphisms as above, then $W\times_X W'\to W$ and
$W\times_X W'\to W'$ are faithfully flat quasi-compact monomorphisms, hence
isomorphisms.

If $X=\Spec A$ is an affine scheme and $Z=\Spec(A/I)$ is a closed subscheme,
then the Zariskification exists and equals $W=\Spec\bigl( (1+I)^{-1}A\bigr)$
\cite[\S2]{raynaud_hensel_rings}.
If $Z=\{x_1,x_2,\dots,x_n\}$ is a finite set of points, then the Zariskification
is the semi-localization at $Z$.

In the following example we show that the Zariskification at two points of
Hironaka's non-projective proper smooth threefold does not exist.

\begin{example}[Non-existence of Zariskification]\label{EX:Zariskification-1}
Let $X$ be a projective threefold and $c,d$ curves as in \cite[p.~443]{hartshorne}. Let $X'$ be the non-projective proper threefold given by gluing the different blow-ups and let $l_0$, $m_0$, $l'_0$ and $m'_0$ be curves on $X'$ as in \loccit\ and let $P'=l_0 \cap m_0$ and $Q'=l'_0\cap m'_0$.

There is no affine neighborhood containing both $P'$ and $Q'$. We claim that the
generization $E$ of $P'$ and $Q'$ is not pro-open (i.e., not represented by a flat
quasi-compact monomorphism).
For this, we can use Raynaud's criterion for locally factorial schemes~\cite[Cor.~2.7]{raynaud_sem-samuel}. Hence, it is enough to show that there is a point $x'$ not in $E$ such that every divisor containing $x'$ intersects $E$ (i.e., intersects $P'$ or $Q'$).
But since $l_0+m'_0$ is numerically trivial, every divisor that intersects $l_0$ properly contains $m'_0$. In particular, every divisor intersecting $l_0$ contains either $P'$ or $Q'$.
Raynaud's criterion is thus not satisfied for a point $x'$ on $l_0$ (not equal to $P'$).
\end{example}

\begin{example}[Algebraic space without Zariski-localization at a point]
For a suitable choice of $X$ and curves $c,d$, one can endow Hironaka's proper
threefold $X'$ with a free action of $G=\ZZ/2\ZZ$ that interchanges $P'$ and $Q'$.
The quotient $X'/G$ is then not a scheme since the image of $\{P',Q'\}$ is a
point $z$ that does not admit an affine neighborhood. Moreover, the
Zariskification at $z$ does not exist. Indeed, if there is a flat
monomorphism $W\to X'/G$ of algebraic spaces with image the generization of
$z$, then it pulls-back to a flat monomorphism $W'\to X'$ with
image the generization of $P'$ and $Q'$. Since $W'$ is a scheme (see
\spcite{0B8A} or \cite[Thm.~3.1.5]{temkin-tyomkin_Prufer}), this contradicts
\Cref{EX:Zariskification-1}.
\end{example}

\end{section}


\bibliography{etale-neighborhoods}
\bibliographystyle{dary}

\end{document}